\documentclass[11pt]{scrartcl}

\usepackage{amsfonts,amsmath,amsthm,amssymb}
\usepackage{mathrsfs,mathtools}
\usepackage{enumerate}
\usepackage{enumitem}
\usepackage{hyperref}
\usepackage{esint}
\usepackage{graphicx}
\usepackage{bm}
\usepackage{commath}
\usepackage{esint}
\usepackage{color}

\usepackage[bw,framed]{includes/mcode}
\usepackage{graphicx,color}
\graphicspath{{figs/}}
\def\R{\mathbb{R}}
\def\C{\mathbb{C}}
\def\N{\mathbb{N}}

\def\Z{\mathbb{Z}}
\def\T{\mathbb{T}}

\def\cC{\mathcal{C}}

\def\cS{\mathcal{S}}

\def\a{\alpha}
\def\b{\beta}

\def\l{\lambda}

\def\p{\partial}

\def\veps{\varepsilon}

\def\vphi{\varphi}

\def\hf{\widehat{f}}
\def\tf{\widetilde{f}}

\def\hu{\widehat{u}}

\newcommand{\dv}[1]{\,{\mathrm d}#1}

\newcommand{\wcheck}[1]{#1\hspace{-.8ex}\mbox{\huge {\lower.45ex \hbox{$\textstyle \check{}$}}} \hspace{.5ex}}

\let\oldmarginpar\marginpar
\renewcommand\marginpar[1]{
  \oldmarginpar[\raggedleft\footnotesize #1]
  {\raggedright\footnotesize #1}}
\usepackage{ifthen}
\usepackage[normalem]{ulem}

\newtheorem{definition}{Definition}
\newtheorem{lemma}[definition]{Lemma}
\newtheorem{theorem}[definition]{Theorem}
\newtheorem{remark}[definition]{Remark}
\definecolor{tourquoise}{RGB}{0,170,180}	
\definecolor{darkred}{RGB}{238,34,34}		
\definecolor{darkgreen}{RGB}{0,190,0}		
\definecolor{lightgray}{RGB}{210,210,210}	
\definecolor{deepblue}{RGB}{0,0,240}		
\definecolor{darkgray}{RGB}{144,144,144}	
\definecolor{kingblue}{RGB}{64,96,224}		
\definecolor{gold}{RGB}{240,208,0}		
\definecolor{verydarkred}{RGB}{176,0,0}		

\usepackage{boxedminipage}
\usepackage{accents}
\def\T{\mathbb{T}}
\def\hv{\widehat{v}}
\def\hw{\widehat{w}}
\def\tv{\widetilde{v}}
\def\tV{\widetilde{V}}
\def\hV{\widehat{V}}
\def\i{\mathrm{i}}
\def\ikx{{\i k\cdot x}}

\def\jinNnd{{j\in\N_n^d}}
\def\kinZnd{{k\in\Z_n^d}}
\def\smz{{\setminus \{0\}}}
\newcommand{\sbring}[1]{\accentset{\circ}{#1}}
\def\oH{\sbring{H}}

\def\Lip{\text{Lip}}
\def\hpsi{\widehat{\psi}}

\newcommand{\abssec}[1]{\noindent\normalsize {\bfseries #1\quad }\ignorespaces}
\renewenvironment{abstract}{\abssec{Abstract}}{\par\vspace{.1in}}
\newenvironment{keywords}{\abssec{Key Words}}{\par\vspace{.1in}}
\newenvironment{AMS}{\abssec{AMS subject
		classification}}{\par\vspace{.1in}}

\DeclareMathAlphabet{\mathpzc}{OT1}{pzc}{m}{it}

\usepackage[textsize=small]{todonotes}
\numberwithin{equation}{section}

\title{Spectral approximation of fractional PDEs in image processing and phase field modeling
\thanks{The work of the first author is partially supported by  NSF grant DMS-1521590. }}

\author{Harbir Antil\thanks{Department of Mathematical Sciences, George Mason University, Fairfax, VA 22030, USA. \texttt{hantil@gmu.edu}}
\and
S\"oren Bartels\thanks{Department of Applied Mathematics, Mathematical Institute, University of 
Freiburg, Hermann-Herder-Str 9, 79104 Freiburg I. Br. Germany.
\texttt{bartels@mathematik.uni-freiburg.de}}
}

\begin{document}

 \maketitle 
  
\begin{abstract}
Fractional differential operators provide an attractive mathematical tool 
to model effects with limited regularity properties. Particular examples
are image processing and phase field models in which jumps across lower
dimensional subsets and sharp transitions across interfaces are of 
interest. The numerical solution of corresponding model problems via
a spectral method is analyzed. Its efficiency 
and features of the model problems are illustrated by numerical experiments. 
\end{abstract}



\section{Introduction}\label{s:intro}

Let $\T^d$, $d\ge 1$, be the $d$-dimensional torus. The purpose of 
this paper is to study the approximation of problems involving the fractional 
Laplace operator of order $2s$ 
\[
 (-\Delta)^s \equiv (-\Delta_{\mathbb{T}^d})^s 
\]
using the Fourier spectral method and to illustrate the importance of fractional 
differential operators. Such operators appear in various 
models with periodic boundary conditions, 
see \cite{MR2773101,MR2276260,MR3477397}. The approach discussed here extends to 
problems with other boundary conditions such as Dirichlet or Neumann 
boundary conditions. For a discussion on nonhomogeneous boundary conditions we refer to \cite{antil2017fractional}.

Motivated by applications including fracture mechanics and turbulence,
see~\cite{ge2015fractional,wow,NEGRETE}, 
problems with fractional derivatives have recently gained a lot of interest. 
Several experiments suggest the presence of fractional derivatives, for instance, 
the electrical signal propagation in a cardiac tissue \cite{AOBueno_DKay_VGrau_BRodriquez_KBurrage_2014a}. 
The appearance of fractional derivatives there is attributed 
to the heterogeneity of the underlying medium. 
A question arises: can one, for example, tailor the diffusion coefficient in  
\cite{AOBueno_DKay_VGrau_BRodriquez_KBurrage_2014a} to get the same effects as 
the fractional model? 
Indeed to arrive at a direct evidence justifying the presence of fractional 
derivatives is a difficult question to address. This paper is an attempt to 
partially address this question by considering two specific applications where the 
presence of spectral fractional operators makes a significant difference.
In particular, we illustrate the effect and advantages of fractional 
derivatives on two, by now, classical problems: image denoising and phase field 
modeling. 

A well-known total variation based image denoising model is the so-called 
Rudin--Osher--Fatemi (ROF) model \cite{rudin1992nonlinear} which seeks a 
minimizer $u\in BV(\T^d)\cap L^2(\T^d)$ for 
\begin{equation}\label{eq:E}
 E(u) = |Du|(\mathbb{T}^d) + \frac{\alpha}{2} \|u-g\|^2.
\end{equation}
Here $\T^d$ denotes the image domain, $\|\cdot\|$ is the norm in $L^2(\T^d;\C)$ 
with corresponding inner product $(\cdot,\cdot)$, and $\alpha>0$ is a 
regularization parameter. The function $g:\T^d\rightarrow \C$ represents the 
given observed possibly noisy image. The first term in $E$ is the total 
variation $|Du|(\T^d)$ which has a regularizing effect but at the same time
allows for discontinuities which may represent edges in the image. The second 
term is the fidelity term which measures the distance to the given image. 
Often, weaker norms such as the $H^{-1}$ norm are considered to define the latter
term. While the existence and uniqueness of a minimizer can be established via
the direct method of calculus of variations, the non-differentiability of the 
total variation is challenging from a computational point of view. In fact, a 
non-exhaustive list of papers that have attempted to resolve this are 
\cite{MR2970738,MR3483094,MR1668254,MR2068672,MR2566128,MR2496060,MR1780606,MR3309171}. 
Another question to ask is, whether natural images belong to $BV(\T^d)\cap L^2(\T^d)$.
The paper \cite{MR1871413} shows that natural images are incompletely  
represented by $BV(\T^d)$ functions. We will handle both these shortcomings 
by replacing the total variation term in \eqref{eq:E} by a squared fractional
Sobolev norm. In other words, we propose to minimize 
\begin{equation}\label{eq:EM}
E(u) = \frac{1}{2}\|(-\Delta)^{s/2}u\|^{2} 
  + \frac{\alpha}{2} \|(-\Delta)^{-\b/2}(u-g)\|^2 ,
\end{equation}
with $0 < s <1$ and $\b\in [0,1]$. The first order necessary and 
sufficient optimality condition determines the unique minimizer $u$ via 
\begin{equation}\label{eq:EMO}
 (-\Delta)^s u + \alpha (-\Delta)^{-\b} (u-g) = 0 \quad \mbox{in } \mathbb{T}^d ,
\end{equation}
which is a linear elliptic partial differential equation (PDE) that 
can be efficiently solved using, for instance, the Fourier spectral method 
(which is the focus of this paper) or the so-called Caffarelli-Silvestre 
extension (in $\R^n$) \cite{LCaffarelli_LSilvestre_2007a} and the Stinga-Torrea extension (in bounded domains) \cite{MR3477397,PRStinga_JLTorrea_2010a}, see also \cite{RHNochetto_EOtarola_AJSalgado_2014a}. Our experiments reveal that the fractional model~\eqref{eq:EMO}
leads to results which are comparable to those provided by the ROF model
but at a significantly reduced computational effort (see Section~\ref{s:Image}). 
We remark that fractional derivatives have been used in 
image processing before; see \cite{carasso2012fractional,carasso2014recovery},
where the authors use spectral methods, and \cite{GH:14}, where the authors 
use the finite element method. However, in both these cases the authors solve 
a fractional dynamical system with initial condition given by $g$. 
For completeness, we also refer to \cite{MR3321062} where the authors consider a
fractional norm equivalent regularization in the context of optimal control problems 
and parameter identification problems -- this equivalent norm was realized using a 
mutilevel approach.

A mathematical justification of our choice of \eqref{eq:EM} as a substitute 
to \eqref{eq:E} is given next. We seek $u$ solving \eqref{eq:EMO} in a 
fractional Sobolev space $H^s(\T^d)$. Moreover, we notice that if 
$g \in L^\infty(\T^d)$ then following Theorem 3.5 part (1)(b) of \cite{MR3385190}
it is possible to show that $u \in L^\infty(\T^d)$, see also 
\cite{antil2016note}. 
We will next see that $BV(\T^d) \cap L^\infty(\T^d)$ 
is contained in $H^s(\T^d)$ for $s < 1/2$. Indeed by Lemma~38.1 
of \cite{MR2328004}  
we have the following continuous embedding 
\[
 BV(\T^d) \cap L^\infty(\T^d) \subset
 B^{1/2}_{2,\infty}(\T^d)
\]
where $B^{1/2}_{2,\infty}(\T^d)$ is a Besov space. In addition, using 
Proposition~1.2 of \cite{MR2231013}, see also \cite[pg.~1222]{MR2175433} 
and \cite[Section~3]{toft2001embeddings}, we have the following continuous embedding 
\[
 B^{1/2}_{2,\infty}(\T^d) \subset H^s(\T^d)
\]
provided that $s < 1/2$. Finally, combining the inclusions we arrive at 
\[
 BV(\T^d) \cap L^\infty(\T^d) \subset H^s(\T^d),
\]
which justifies our energy functional~\eqref{eq:EM}.
We remark that the regularizing quadratic term in~\eqref{eq:EM} does not have the 
gradient sparsity property of the total variation norm. 
This effect however cannot be proven for the ROF model in general due the 
presence of the quadratic fidelity term but is certainly visible in experiments.

As a second example we consider gradient flows of the energy functional 
\begin{equation}\label{eq:Eeps}
 E_\veps(u)=\frac{1}{2}\|(-\Delta)^{s/2} u\|^{2} + \veps^{-2} \int_{\T^d} F(u)\;dx
\end{equation}
with initial condition $u(0)=u_0$. The $L^2$-gradient flow of \eqref{eq:Eeps} 
leads to the fractional Allen--Cahn equation 
\begin{align}\label{eq:fracAC}
\begin{aligned}
\partial_t u + (-\Delta)^s u + \veps^{-2} f(u) 
  &= 0 \; \quad  \mbox{in } (0,T) \times \T^d, \\ 
 u(0,\cdot) &= u_0 \quad  \mbox{in } \T^d ,
\end{aligned}
\end{align}
where $0 < s < 1$, $T>0$, and $f=F'$ is typically 
nonlinear in $u$. Moreover, 
\[
 \veps^{-2} := \left\{ \begin{array}{ll} 
                \widetilde\veps^{-2s}  & \mbox{if }  s\in (0,1/2) ,\\
                |\log \widetilde\veps| & \mbox{if }  s=1/2 ,\\    
                \widetilde\veps^{1-2s} & \mbox{if }  s\in (1/2,1)   
                \end{array}
              \right.
\]
where $0<\widetilde\veps<1$. When $s=1$, we set $\veps^{-2} := \widetilde\veps^{-2}$.
The $H^{-\alpha}$ gradient flow, with $\alpha\in(0,1]$ leads 
to the fractional Cahn--Hilliard equation
\begin{align}\label{eq:fracCH}
\begin{aligned}
(-\Delta)^{-\a}  \partial_t u + (-\Delta)^{s} u + \veps^{-2}  f(u) 
  &= 0 \quad \; \mbox{in } (0,T) \times \T^d, \\ 
 u(0,\cdot) &= u_0 \quad \mbox{in } \T^d ,
\end{aligned}
\end{align}
By testing \eqref{eq:fracCH} with a constant function it is easy to check
that \eqref{eq:fracCH} is mass conserving. 
We note that throughout this
article we consider $\widetilde\veps$ as a fixed small number. 
The aforementioned
scaling of $\widetilde\veps$ is the right scaling to obtain a sharp interface 
limit as $\widetilde\veps \to 0$ we refer the reader to~\cite{MR2948285}. 
We remark that even though
the optimality system in case of image denoising is linear \eqref{eq:EMO}, the
system for the fractional phase field model is nonlinear \eqref{eq:fracCH}
and controlling these nonlinearities in the presence of fractional derivatives
turned out to be a nonobvious task.

When $s=1$, a standard numerical method requires a fine mesh resolution 
around interfaces to capture sharp transitions \cite{MR2764423}. 
The fully discrete scheme proposed in this paper is unconditionally stable 
and supported by a rigorous error analysis. In our experiments 
we observe that using the 
spectral method and choosing small values for $s$, it is possible to obtain
sharp interfaces on relatively coarse meshes and moderate values for 
$\widetilde\veps$.

We remark that the use of spectral methods in the context of phase field models 
(when $s=1$, $\alpha=0$ or $\alpha=1$) has been considered before, 
see \cite{MR3225505,MR2728977}. We further remark that the recent paper 
\cite{MR3491939} also investigates a fractional Allen--Cahn equation and uses the 
fractional Riemann-Liouville derivative which is different from our definition. 
In addition, few analytical details are provided. We also refer 
to \cite{ainsworthanalysis} which discusses analytical properties
of a fractional Cahn--Hilliard equation with fixed $s=1$. The authors report 
that the dynamics in case $\alpha>0$ and $s=1$ are closer to the classical 
Cahn-Hilliard equation than to the Allen-Cahn equation. The error analysis 
provided there is restricted to spatial discretizations while the used
fully discrete 
scheme treats the nonlinearity explicitly and is hence only conditionally stable. 

We remark that the goal of this article is to show possible applications 
of fractional PDEs. The simple image denoising problem serves as a
model problem in which the effect of fractional derivatives of different
order becomes directly apparent. The nonlinear evolution model defined by
the fractional phase field equation combines different effects so that 
an interpretation of the effect of fractional derivatives of different order
requires a more careful interpretation. Our experiments are meant to illustrate
these effects for different parameters $s$.

This paper is organized as follows: In Section~\ref{s:notation} we recall 
facts about spectral interpolation estimates in fractional Sobolev spaces. 
The fractional Laplace operator and its discretization by the spectral method 
are addressed in Section~\ref{s:fLap}. We present the details on the fractional 
image denoising problem in Section~\ref{s:Image}. Section~\ref{s:fracAC} is 
devoted to a general error analysis for a numerical scheme covering both 
the fractional  Allen--Cahn and Cahn--Hilliard equations. We conclude with several 
illustrative numerical examples in Section~\ref{s:numerics}.

\smallskip

\section{Spectral approximation} \label{s:notation}

In this section we specify notation needed to define the 
discrete Fourier transformation and recall elementary approximation results
in fractional Sobolev spaces. 

\subsection{Discrete Fourier transformation}\label{s:discF}
We consider the $2\pi$-periodic torus $\T^d$ and the  
set of grid points $(x_j: j\in \N_n^d)$ on $\T^d$ defined by 
$
x_j = (j_1,\dots,j_d) \frac{2 \pi}{n},
$
where 
$
\N_n^d = \big\{ j = (j_1,\dots,j_d) \in \Z^d: 0 \le j_i \le n-1 \big\}.
$
A family of grid functions $(\Phi^k:k\in \Z_n^d)$ is defined by 
\[
\Phi^k = \big(e^{\ikx_j} : j\in \N_n^d\big),
\]
where
$
\Z_n^d = \big\{ k = (k_1,\dots,k_d) \in \Z^d: -n/2 \le k_i \le n/2-1 \big\}
$
and $\i^2 =-1$.
For grid functions $V=(v_j:j\in \N_n^d)$ and $W=(w_j:j\in \N_n^d)$ we
define the discrete scalar product 
\[
(V,W)_n = \frac{(2\pi)^d}{n^d} \sum_\jinNnd v_j \overline{w}_j.
\]
The associated norm is denoted $\|\cdot\|_n$. Notice that the family 
$(\Phi^k:k\in \Z_n^d)$ defines an orthogonal basis for the
space of grid functions with $\|\Phi^k\|_n = (2\pi)^{d/2}$. 
The {\em discrete Fourier transform} of a grid function $V=(v_j:j\in \N_n^d)$ 
is the coefficient vector $\tV = (\tv_k:\kinZnd)$ with 
\[
\tv_k = (V,\Phi^k)_n.
\] 
With these  coefficients we have
$V = (2\pi)^{-d} \sum_\kinZnd \tv_k \Phi^k.$

\subsection{Trigonometric interpolation}\label{s:trigI}
We consider the space of trigonometric polynomials defined via
\[
\cS_n = \big\{ v_n \in C(\T^d;\C): v_n(x) 
= \sum_\kinZnd c_k \vphi^k(x),  \, c_k \in \C \big\},
\]
with the functions $\vphi^k(x) = e^\ikx$ which define an orthogonal 
basis for $\cS_n$ with respect to the inner product in $L^2(\T^d;\C)$. 
With $v \in C(\T^d;\C)$ we associate a grid function 
$V = (v_j:\jinNnd)$ via $v_j = v(x_j)$, $j\in \N_n^d$. Notice that 
for $v_n,w_n\in \cS_n$ with associated grid functions $V,W$ we have
\[
(v_n,w_n) = (V,W)_n.
\] 
The discrete Fourier transformation gives rise to a nodal 
interpolation operator.

\begin{definition}
Given $v\in C(\T^d;\C)$ with nodal values $V=(v_j:\jinNnd)$
and discrete Fourier coefficients $\tV = (\tv_k:\kinZnd)$, 
the {\em trigonometric interpolant} $I_n v\in \cS_n$ of $v$ 
is defined via
\[
I_n v = \frac{1}{(2\pi)^d}\sum_\kinZnd \tv_k \vphi^k.
\]
\end{definition}

\begin{remark} 
(i) Note that $I_nv(x_j) = v(x_j)$ for all $\jinNnd$. \\
(ii) We have $\tv_k = (v,\vphi^k)_n$ for all $\kinZnd$. \\
(iii) We have $(I_n v,w_n)_n = (v,w_n)_n$
for all $w_n\in \cS_n$ and $v\in C(\T^d;\C)$. 
\end{remark}

The {\em (continuous) Fourier transform} of a function $v \in L^2(\T^d;\C)$
is the coefficient vector $\hV = (\hv_k: k\in\Z^d)$ defined by 
\[
\hv_k = (v,\vphi^k).
\]
Note that here the $L^2$ inner product is used instead of its discrete
approximation. 
With respect to convergence in $L^2(\T^d;\C)$ we have that 
$v = (2\pi)^{-d}\sum_{k\in \Z^d} \hv_k \vphi^k,$ 
and, in particular, 
Plancherel's formula 
$(v,w) = (2\pi)^{-d} (\hv,\widehat{w})_{\ell^2(\Z^d)}$.

\subsection{Approximation in Sobolev spaces}\label{s:sob}
We analyze the approximation properties of the interpolation operator
$I_n$ in terms of Sobolev norms and with the help of the $L^2$ projection
onto $\cS_n$ which is obtained by truncation of the Fourier series of 
a function. 

\begin{definition}\label{def:Pn}
The {\em $L^2$ projection} $P_n:L^2(\T^d;\C)\to \cS_n$ is for $v\in L^2(\T^d;\C)$
defined by the condition that for all $w_n\in \cS_n$ we have
\[
(P_n v, w_n) = (v,w_n).
\]
\end{definition}

Note that for every $v\in L^2(\T^d;\C)$ we have 
$P_n v = \sum_\kinZnd \hv_k \vphi^k$. The following definition is motivated
by the fact that $(\widehat{\p^\a v})_k = \i^{|\a|} k^\a \hv_k$ for every
$\a\in \N_0^d$. 

\begin{definition}
Given $\mu \ge 0$ the {\em Sobolev space} $H^\mu (\T^d;\C)$ 
consists of all functions $v\in L^2(\T^d;\C)$ with 
\[
|v|_\mu^2 = \sum_{k\in\Z^d } |k|^{2\mu} |\hv_k|^2 < \infty.
\]
Its dual $H^{-\mu}(\T^d;\C)$ consists of all linear functionals 
$\psi: H^\mu(\T^d;\C) \to \C$ with 
\[
|\psi|_{-\mu}^2 = \sum_{k\in\Z^d\setminus \{0\}} |k|^{-2\mu} |\hpsi_k|^2 < \infty,
\]
where $\hpsi_k = \psi(\phi_k)$. 
\end{definition}

The Sobolev spaces allow us to quantify approximation properties of 
the operators $P_n$ and $I_n$. We refer the reader to Chapter 8 in~\cite{MR1870713} for 
details. 

\begin{lemma}[Projection error]\label{lem:projerr}
For $\l,\mu\in \R$ with $\l\le \mu$ and $v\in H^\mu(\T^d;\C)$ we have 
\[
|v-P_n v|_\l \le \Big(\frac{n}{2}\Big)^{-(\mu-\l)} |v|_\mu . 
\]
\end{lemma}

By comparing $P_n$ and $I_n$ we obtain a trigonometric interpolation 
estimate. It is shown in Remark~8.3.1 of~\cite{MR1870713} that the conditions of the following
result cannot be improved in general. 

\begin{lemma}[Interpolation error]\label{lem:interpE}
If $\mu>d/2$, $0\le \l\le \mu$, and $v\in H^\mu(\T^d;\C)$ we have
\[
|v-I_nv|_\l \le c_{d,\l,\mu} \Big(\frac{n}{2}\Big)^{-(\mu-\l)}
|v|_\mu
\]
with a constant $c_{d,\l,\mu}>0$ that is independent of $v$ and $n$.
\end{lemma}

We conclude the section with an inverse estimate. Particularly,
for every function $v_n \in \cS_n$ and $r\ge s$ we have that
\begin{equation}\label{eq:inv_est}
|v_n|_r \le \max_{\kinZnd} |k|^{r-s} |v_n|_s \le \Big(\frac{n}{2}\Big)^{r-s} |v_n|_s.
\end{equation}

\section{Fractional Laplace operator}\label{s:fLap}
We define subspaces of Sobolev spaces via
\[
\oH^r(\T^d;\C) = \big\{v\in H^r(\T^d;\C): 
\hv_0 = 0\big\}.
\]
For $r\ge 0$ the subspaces consist of Sobolev functions with
vanishing mean. On the subspaces $\oH^r(\T^d;\C)$ the corresponding seminorms 
$|\cdot|_r$ are norms. 

\begin{definition} 
For $s,\mu \ge 0$ and $v\in H^\mu(\T^d;\C)$ 
the {\em fractional Laplacian} of $v$ is the  
(generalized) function $(-\Delta)^s v \in \oH^{\mu-2s}(\T^d;\C)$ defined by
\[
(-\Delta)^s v = \frac{1}{(2\pi)^d} \sum_{k\in\Z^d \smz} |k|^{2s} \hv_k \vphi^k.
\]
\end{definition}

Given $f\in L^2(\T^d;\C)$ with vanishing mean 
the {\em fractional Poisson problem} seeks $u\in \oH^{s}(\T^d;\C)$ with 
\begin{equation}\label{fPcontPDE}
(-\Delta)^s u = f.
\end{equation}
The unique solution to \eqref{fPcontPDE} is given by
\begin{equation}\label{eq:fPcont}
u = \frac{1}{(2\pi)^d}\sum_{k\in\Z^d \smz} |k|^{-2s} \widehat{f}_k \vphi^k,
\end{equation}
and in fact satisfies $u\in \oH^{2s}(\T^d;\C)$. More generally, 
for $f\in \oH^\mu(\T^d;\C)$ we have 
\[
|u|_{\mu+2s} = |f|_\mu,
\]
i.e., the fractional Laplacian defines an isometric isomorphism
\[
(-\Delta)^s : \oH^r(\T^d;\C) \to \oH^{r-2s}(\T^d;\C).
\]
We define the fractional Laplace operator of negative order
as the inverse of $(-\Delta)^s$, i.e., 
\[
(-\Delta)^{-s} = \big((-\Delta)^s\big)^{-1} : 
\oH^r(\T^d;\C) \to \oH^{r+2s}(\T^d;\C).
\]
Note that for $r,s\in \R$ and $v\in \oH^r(\T^d;\C)$ with 
$(-\Delta)^s v\in L^2(\T^d;\C)$ we have 
\[
|v|_s = \|(-\Delta)^{s/2} v\|.
\]
If $s\le r$ we have the continuous embedding property 
\begin{equation}\label{eq:cont_embed}
|v|_s \le |v|_r.
\end{equation}

The discretized fractional Poisson problem seeks for a given
$f_n \in \cS_n$ with vanishing mean a function $u_n \in \cS_n$ with
\begin{equation}\label{fPdiscPDE}
(-\Delta)^s u_n = f_n.
\end{equation}
The uniquely defined solution is given by 
\begin{equation}\label{eq:fPdisc}
u_n = \frac{1}{(2\pi)^d} \sum_{\kinZnd \smz} |k|^{-2s} \tf_k \vphi^k . 
\end{equation}
There are two noticeable differences between the continuous 
\eqref{eq:fPcont} and the discrete solutions \eqref{eq:fPdisc}. 
Besides the finite and the infinite sums, $u_n$ contains the discrete 
Fourier coefficients $\tf_k$ and $u$ contains the continuous Fourier 
coefficients $\widehat{f}_k$. The following a priori error estimates
hold. 

\begin{theorem}\label{thm:fracPest}
Let $u$ and $u_n$ solve the continuous \eqref{fPcontPDE} and the 
discrete \eqref{fPdiscPDE} problems, respectively. We have 
\[
|u-u_n|_s \le |f-f_n|_{-s}.
\]
In particular, if $f\in \oH^\mu(\T^d;\C)$ and $f_n = P_nf$ we have 
\[
|u-u_n|_s \le \Big(\frac{n}2\Big)^{-(\mu+s)} |f|_\mu, 
\]
while if $f \in \oH^\nu(\T^d;\C)$ with $\nu > d/2$ and $f_n = I_n f$ 
we have 
\[
|u-u_n|_s \le c_{d,0,\nu} \Big(\frac{n}{2}\Big)^{-\nu} |f|_\nu.
\]
\end{theorem}

\begin{proof}
In view of \eqref{fPcontPDE} and \eqref{fPdiscPDE} we have 
\begin{align*}
|u-u_n|_s^2 &= \big((-\Delta)^s (u-u_n), u-u_n \big) \\
&= \big(f-f_n, u-u_n\big)  \\
&\le |f-f_n|_{-s} |u-u_n|_s.
\end{align*}
This implies the general estimate and in combination with 
Lemma~\ref{lem:projerr} the estimate in case $f_n = P_nf$.
With~\eqref{eq:cont_embed}
 and Lemma~\ref{lem:interpE} we deduce that 
\[
|f-I_nf|_{-s} \le \|f-I_nf\|
\le c_{d,0,\nu} \Big(\frac{n}{2}\Big)^{-\nu} |f|_\nu
\]
which implies the estimate. 
\end{proof}

\section{Fractional image denoising}\label{s:Image}
Our second problem is a replacement of the ROF image denoising model \eqref{eq:E}. 
Given an image $g \in L^2(\mathbb{T}^d;\C)$ we propose to minimize 
\begin{equation}\label{eq:EMI}
 E(u) = \frac{1}{2}\|(-\Delta)^{s/2}u\|^{2} + \frac{\a}{2} \|(-\Delta)^{-\b/2}(u-g)\|^2 ,
\end{equation}
with $0 < s <1$ and $\b\in [0,1]$. The minimization is carried out over $H^s(\T^d;\C)$
when $\beta=0$ and over $\oH^s(\T^d;\C)$ when $\b$ is positive. In the latter case 
we assume that $g$ has a vanishing mean.
The existence and uniqueness of a minimizer follows by using the direct method 
in the calculus of variations. The first order necessary and 
sufficient optimality condition determines the unique minimizer $u$ via 
\begin{equation}\label{eq:EMOI}
 (-\Delta)^s u + \a (-\Delta)^{-\b} u = \a (-\Delta)^{-\b} g \quad \mbox{in } \mathbb{T}^d .
\end{equation}
We note that since $g\in L^2(\T^d;\C)$, we have $u \in H^{2(s+\b)}(\T^d;\C)$.
In particular, the solution to \eqref{eq:EMOI} is
\[
u = \frac{\a}{(2\pi)^d}\sum_{k\in\Z^d} 
\big(|k|^{2(s+\b)}+\a \big)^{-1}  \widehat{g}_k \vphi^k.
\]

The discretized problem seeks for a given $g_n \in \cS_n$ a 
function $u_n \in \cS_n$
with 
\begin{equation}\label{fPdiscPDE_image}
(-\Delta)^s u_n + \a (-\Delta)^{-\b} u_n = \a (-\Delta)^{-\b} g_n .
\end{equation}
The uniquely defined solution is given by
\[ 
u_n = \frac{\alpha}{(2\pi)^d}\sum_{\kinZnd} 
\big(|k|^{2(s+\b)}+\alpha \big)^{-1}  \widetilde{g}_k \vphi^k   
\]

\begin{theorem}
Let $u$ and $u_n$ solve the continuous and the discrete problems
\eqref{fPdiscPDE_image} and \eqref{eq:EMOI}, respectively. We have that
\[
|u-u_n|_s^2 + \frac{\a}{2} |u-u_n|_{-\b}^2 \le \frac{\a}{2} |g-g_n|_{-\b}^2.
\]
In particular, if $g\in H^\mu(\T^d;\C)$ and $g_n = P_ng$ we have 
\[
|u-u_n|_s + (\a/2)^{1/2} |u-u_n|_{-\b}
 \le \a^{1/2} \Big(\frac{n}{2}\Big)^{-(\mu+\b)} |g|_\mu,
\]
while if $g \in H^\nu(\T^d;\C)$ with $\nu > d/2$ and $g_n = I_n g$ 
we have 
\[
|u-u_n|_s + (\a/2)^{1/2} |u-u_n|_{-\b}
\le \a^{1/2} c_{d,0,\nu} \Big(\frac{n}{2}\Big)^{-\nu} |g|_\nu.
\]
\end{theorem}

\begin{proof}
Testing the difference of~\eqref{eq:EMOI} and~\eqref{fPdiscPDE_image}
by $u-u_n$ implies that
\[\begin{split}
|u-u_n|_s^2 + \a |u-u_n|_{-\b}^2  &= \a \big((-\Delta)^{-\b}(g-g_n),u-u_n\big) \\
&\le \frac{\a}{2} |g-g_n|_{-\b}^2 + \frac{\a}{2} |u-u_n|_{-\b}^2.
\end{split}\]
The estimates follow from using $(a+b)^2 \le 2 (a^2 + b^2)$ and
arguing as in the proof of Theorem~\ref{thm:fracPest}.
\end{proof}

\section{Fractional phase field equations}\label{s:fracAC}

Given parameters $\a,s\ge 0$ we recall the fractional Cahn-Hilliard 
equation \eqref{eq:fracCH} 
\begin{equation}\label{eq:frac_ch}
(-\Delta)^{-\a} \p_t u + (-\Delta)^s u = -\veps^{-2} f(u) 
\end{equation}
on a $d$-dimensional torus $\T^d$ and with initial condition
$u(0)=u_0$. We recall that $\alpha=0$ gives rise to the fractional 
Allen--Cahn equation \eqref{eq:fracAC}. Below we will impose the restrictions $s>0$
and $s\ge \a$. 

We assume a splitting of the nonnegative potential
$F$ into convex and concave parts $F^{cx}$ and $F^{cv}$ which induces 
a decomposition of $f=F'$ into a monotone and an antimonotone 
part 
\[
f = f^{cx} + f^{cv}.
\]
We assume for simplicity that $f^{cx}$ and $f^{cv}$ are smooth
and Lipschitz continuous. The latter condition is justified by 
a maximum principle in the case for the Allen--Cahn equation and  
$L^\infty$ bounds for solutions of the Cahn--Hilliard equation \cite{MR1367359} 
corresponding to $(\a,s)=(0,1)$ and $(\a,s)=(1,1)$, respectively.

\subsection{Numerical scheme and error analysis}

The numerical scheme computes iterates $(u_n^k)_{k=0,\dots,K} \subset \cS_n$ via  
\begin{equation}\label{eq:discr_scheme}
(-\Delta)^{-\a} d_t u_n^k + (-\Delta)^s u_n^k 
+ \veps^{-2} I_n f^{cx}(u_n^k)
= - \veps^{-2} I_n f^{cv}(u_n^{k-1}).
\end{equation}
where $d_t w^k = (w^k-w^{k-1})/\tau$ with $\tau>0$ being the 
time step-size and $u_n^0$ is a suitable approximation of $u^0$. 
By applying the operator $(-\Delta)^\a$ and testing the resulting
identity with constant functions we observe the  
mass conservation property $(d_t u_n^k,1) = 0$ if $\a>0$.
Existence of the iterates is established via convex minimization
problems; if the convex part of $F$ is quadratic then $f^{cx}$ is 
linear and the scheme~\eqref{eq:discr_scheme}
defines a linear system of equations. 
The scheme is unconditionally energy stable in the sense that
we have 
\[
|d_t u_n^k |_{-\a}^2 + \frac{\tau}{2} |d_t u_n^k|_s^2 
+ d_t E_\veps^n(u_n^k) \le 0,
\]
with the discrete energy functional 
\[
E_\veps^n (v_n) = \frac12 \|(-\Delta)^{s/2} v_n\|^2 
+ \veps^{-2} \big(F(v_n),1\big)_n.
\]
This follows from testing~\eqref{eq:discr_scheme} with $d_t u_n^k$, 
using 
\[
\big((-\Delta)^s u_n^k, d_t u_n^k\big) 
= \frac12 d_t \|(-\Delta)^{s/2} u_n^k\|^2
+ \frac{\tau}{2} \|(-\Delta)^{s/2} d_t u_n^k\|^2,
\]
and noting that as a consequence of convexity and concavity we have 
\[\begin{split}
\big(f^{cx}(u_n^k), d_t u_n^k\big)_n & \ge \big(d_t F^{cx}(u_n^k),1\big)_n, \\
\big(f^{cv}(u_n^{k-1}), d_t u_n^k\big)_n & \ge \big(d_t F^{cv}(u_n^k),1\big)_n.
\end{split}\]
Assuming initial data with $E_\veps^n(u_n^0) \le c$ as $n\to \infty$,
the energy estimate provides a~priori bounds on interpolants of the
approximations which allows us to select an accumulation point 
\[
u\in H^1([0,T];H^{-\a}(\T^d))\cap L^\infty([0,T];H^s(\T^d))
\]
as $\tau\to 0$ and $n\to \infty$. Its identification as a solution 
for the fractional Cahn--Hilliard equation follows from the Aubin--Lions
lemma provided that $s>0$. Uniqueness of solutions is a consequence of the 
assumed Lipschitz continuity of $f$. 
For an error analysis we note that $u\in C([0,T];L^2(\T^d))$,
let $u^k = u(t_k)$ with $t_k = k\tau$, and define 
\[
e_n^k = u_n^k - P_n u^k,
\]
where $P_n$ is the orthogonal projection 
given in Definition~\ref{def:Pn}. 
Note that we have $(P_n v,w_n)_n = (P_nv,w_n)$ but $(v,w_n)_n\neq (v,w_n)$
unless $v$ belongs to~$\cS_n$. We omit the subscript $n$ whenever
the scalar product is applied to two functions belonging to $\cS_n$.
For ease of readability we abbreviate
\[
f_\veps = \veps^{-2} f, \quad f_\veps^{cx} = \veps^{-2} f^{cx}, \quad
f_\veps^{cv} = \veps^{-2} f^{cv}.
\]
The sequences $(u_n^k)$ and $(u^k)$ satisfy the discrete equations 
\[\begin{split}
(-\Delta)^{-\a} d_t u_n^k + (-\Delta)^s u_n^k & = -I_n f_\veps (u_n^k) 
- I_n \big(f_\veps^{cv}(u_n^{k-1})-f_\veps^{cv}(u_n^k)\big), \\
(-\Delta)^{-\a} d_t P_n u^k + (-\Delta)^s P_n u^k &= -P_n f_\veps(u^k) 
 + (-\Delta)^{-\a} P_n (d_t u^k -\p_t u^k),
\end{split}\]
where we used that $P_n$ commutes with $(-\Delta)^r$ for every $r\in \R$.
Subtracting the identities leads to the error equation
\[
(-\Delta)^{-\a} d_t e_n^k+ (-\Delta)^s e_n^k  = A_n^k + B_n^k + C_n^k,
\]
with the discretization errors 
\[\begin{split}
A_n^k &= - I_n f_\veps( u_n^k)+ P_n f_\veps(u^k),  \\
B_n^k &= - I_n \big(f_\veps^{cv}(u_n^{k-1})-f_\veps^{cv}(u_n^k)\big), \\
C_n^k &= - (-\Delta)^{-\a} P_n (d_t -\p_t) u^k.
\end{split}\]
Testing the error equation with $e_n^k$ shows that we have
\[\begin{split}
\frac12 d_t |e_n^k|_{-\a}^2 + \frac{\tau}{2} |d_t e_n^k|_{-\a}^2
+ |e_n^k|_s^2 
= (A_n^k,e_n^k) + (B_n^k,e_n^k) + (C_n^k,e_n^k).
\end{split}\]
To bound the first term on the right-hand side we insert $f_\veps(u_n^k)$,
use Lemma~\ref{lem:interpE}, the inverse estimate~\eqref{eq:inv_est},
and insert $P_n u^k$ to deduce with the Lipschitz continuity of $f_\veps$ that 
\[\begin{split}
(A_n^k,e_n^k) 
&= -\big(I_nf_\veps(u_n^k)- f_\veps( u_n^k) + f_\veps (u_n^k)- f_\veps (u^k), e_n^k\big)  \\
&\le c_{d,0,1} n^{-1} |f_\veps (u_n^k)|_1 \|e_n^k\| + |f_\veps|_{\Lip} \|u_n^k - u^k\| \|e_n^k\| \\
&\le c_{d,0,1} n^{-1} |f_\veps|_{\Lip} |u_n^k|_1 \|e_n^k\| 
+ |f_\veps|_{\Lip} \big(\|u_n^k - P_n u^k\| + \|P_n u^k - u^k\| \big)\|e_n^k\| \\
&\le |f_\veps|_{\Lip} \big( c_{d,0,1} n^{-s}  |u_n^k|_s + \|P_nu^k - u^k\|\big) \|e_n^k\|
+ |f_\veps|_{\Lip} \|e_n^k\|^2. 
\end{split}\]
For the second term we have that 
\[
(B_n^k,e_n^k) \le \tau |f_\veps|_{\Lip} \|d_t u_n^k\| \|e_n^k\|
\le \tau n^\a |f_\veps|_{\Lip} |d_t u_n^k|_{-\a} \|e_n^k\|,
\]
where we used the inverse estimate 
$\|w_n\| \le n^\a |w_n|_{-\a}$ if $\a>0$.  
For the third term we assume that $u\in C^2([0,T];H^{-\a}(\T^d))$ and estimate
\[
(C_n^k,e_n^k) = ((-\Delta)^{-\a} [\p_t -d_t] u^k, e_n^k) 
\le \frac{\tau}{2} \sup_{t\in [0,T]} |\p_t^2 u(t)|_{-\a} |e_n^k|_{\a}
\]
A combination of the estimates, multiplication by $\tau$, and 
summation over $k=1,2,\dots,K$, yield that 
\[\begin{split}
& \frac12 |e_n^K|_{-\a}^2 + \tau \sum_{k=1}^K |e_n^k|_s^2 \le \frac12 |e_n^0|_{-\a}^2   
+ K \tau |f_\veps|_{\Lip} \max_{k=1,\dots,K} \big(c_{d,0,1} n^{-s}  |u_n^k|_s +
 \|u^k-P_nu^k\|\big)^2 \\
& \quad  + K \tau^3 n^{2\a} |f_\veps|_{\Lip} \max_{k=1,\dots,K} |d_t u_n^k|_{-\a}^2 
+ K \tau^3 \sup_{t\in [0,T]} |\p_t^2 u(t)|_{-\a}^2 \\
& \quad + 3 |f_\veps|_{\Lip} \tau \sum_{k=1}^K \|e_n^k\|^2
+ \frac{\tau}{4} \sum_{k=1}^K |e_n^k|_\a^2.
\end{split}\]
If $\a=0$ we may directly apply the discrete Gronwall lemma to obtain 
an error estimate. If $\a>0$ we assume $\a \le s$, require that
$(u_n^0,1) = (u^0,1)$ so that $(e_n^k,1)=0$, and use the bound
\[
\|e_n^k\|^2 \le |e_n^k|_{-s} |e_n^k|_s \le |e_n^k|_{-\a} |e_n^k|_s,
\]
to deduce with Young's inequality the estimate 
\[
E_n^K = \frac12 |e_n^K|_{-\a}^2 + \frac{\tau}{2} \sum_{k=1}^K |e_n^k|_s^2  
\le D_0  + D_1 \tau \sum_{k=1}^K E_n^k.
\]
Here, $D_0$ is the sum of the first four terms on the right-hand side of the
above estimate and $D_1 = 2 (3 |f_\veps|_{\Lip})^2$.
The discrete Gronwall lemma leads to the estimate 
\[
E_n^K \le 2 D_0 \exp(D_1 T)
\]
for all $K$ with $K\tau \le T$ provided that $\tau D_1 \le 1/2$. 
With the triangle inequality and approximation estimates for $P_n$
we obtain the following error estimate.

\begin{theorem}
Let $u\in C([0,T];H^s(\T^d)) \cap C^2([0,T];H^{-\a}(\T^d))$ solve~\eqref{eq:fracCH}
and let the sequence $(u_n^k)_{k=0,\dots,K}\subset \cS_n$ be defined 
via~\eqref{eq:discr_scheme}. There exists a constant $c_\veps >0$ such that we have
\[
\max_{k=1,\dots,K} |u(t_k)-u_n^k|_{-\a} \le c_\veps (\tau + \tau n^\a +  n^{-s})
\]
for all $\tau>0$ and $n\in \N$. 
\end{theorem}

The constant $c_\veps$, in general, depends exponentially on $\veps^{-1}$.
For the derivation of this estimate we used the indicated regularity assumption. 
By standard arguments, see~\cite{MR1479170} the regularity assumption on
the exact solution can be weakened to the conditions
\[
u\in L^\infty([0,T];H^s(\T^d)), \quad \p_t^2 u \in L^2([0,T];H^{-\a}(\T^d)).
\]
The suboptimal term $\tau n^\a$ corresponds to the semi-implicit
treatment of the nonlinearity which makes the scheme~\eqref{eq:discr_scheme}
fully practical. 

\subsection{Improved error estimate via spectral bounds}

A significantly improved error estimate can be obtained if additional 
analytical knowledge about the evolution is available, e.g., in the form of
lower bounds for the principal eigenvalue 
\[
\l(t) = \min_{v \in H^s(\T^d)}
\frac{| v|_s^2 + \veps^{-2} \big(f'(u(t))v,v\big)}
{|v|_{-\a}^2}.
\]
For ease of presentation we only consider the fractional Allen--Cahn
equation with $\a=0$ and outline the main arguments 
following~\cite{MR1971212,MR2764423}. We focus on the 
contribution to the error equation resulting from the nonlinearity 
and write it with abstract consistency functionals $\cC_n^k$ as 
\[
d_t e_n^k+ (-\Delta)^s e_n^k  
= \cC_n^k + \veps^{-2} P_n \big(f(u^k)-f(u_n^k)\big).
\]
A precise formula for $\cC_n^k$ is obtained from subtracting the
projected partial differential equation evaluated at $t_k$
onto $\cS_n$ from the numerical scheme as above, e.g., in case of
a fully implicit numerical scheme with a nodal interpolation of
the nonlinear term we have
\[\begin{split}
\cC_n^k &= -\veps^{-2} \big(P_n f(u_n^k)- I_n f(u_n^k)\big)
+ \veps^{-2} P_n f(u^k) + d_t P_n u^k + (-\Delta)^s P_n u^k \\
&= -\veps^{-2} \big(P_n f(u_n^k)- I_n f(u_n^k)\big) + P_n \big(d_t u^k - \p_t u(t_k)\big).
\end{split}\]
To relate the error equation to the principal eigenvalue we require  
a controlled failure of monotonicity for $f$ in the sense that 
there exists a constant $c_f>0$ such that for all $a,b\in \R$ we have
\[
\big(f(a)-f(b)\big)(a-b) \ge f'(a)(a-b)^2 - c_f |a-b|^3.
\]
With this estimate we deduce with $c_f' = \|f'\|_{L^\infty(\R)}$  
that 
\[\begin{split}
\frac12 d_t & \|e_n^k\|^2 + \frac{\tau}{2} \|d_t e^n_k\|^2 + |e_n^k|_s^2 \\
& = -\veps^{-2} \big(f(u^k) - f(u_n^k),e_n^k\big) + (\cC_n^k,e_n^k) \\
&\le -\veps^{-2} \big(f'(u^k) e_n^k,e_n^k \big) + c_f \veps^{-2} \|e_n^k\|_{L^3(\T^d)}^3
+ \frac{\veps^{-2}}{2} |\cC_n^k|_{-s}^2 + \frac{\veps^2}{2} |e_n^k|_s^2 \\
&\le - (1-\theta) \veps^{-2}  \big(f'(u^k) e_n^k,e_n^k \big) 
+ \theta \veps^{-2} c_f'\|e_n^k\|^2 \\
& \qquad + c_f \veps^{-2} \|e_n^k\|_{L^3(\T^d)}^3  
+ \frac{\veps^{-2}}{2} |\cC_n^k|_{-s}^2 + \frac{\veps^2}{2} |e_n^k|_s^2.
\end{split}\]
We incorporate the eigenvalue $\l^k = \l(t_k)$ via 
\[
-\veps^{-2} \big(f'(u^k)e_n^k,e_n^k\big) \le |e_n^k|_s^2  - \l^k \|e_n^k\|^2.
\]
Choosing $\theta = \veps^2$ and letting $\mu^k = \max\{-\l^k,0\}$
thus leads to the estimate 
\[\begin{split}
\frac12 d_t  \|e_n^k\|^2 + |e_n^k|_s^2 \le & (1-\veps^2) |e_n^k|_s^2  + \mu^k \|e_n^k\|^2 \\
&+  c_f' \|e_n^k\|^2 + c_f \veps^{-2} \|e_n^k\|_{L^3(\T^d)}^3
+ \frac{\veps^{-2}}{2} |\cC_n^k|_{-s}^2 + \frac{\veps^2}{2} |e_n^k|_s^2.
\end{split}\]
Rearranging terms gives 
\[\begin{split}
d_t & \|e_n^k\|^2 + \veps^2 |e_n^k|_s^2 
\le 2 \big(\mu^k + c_f') \|e_n^k\|^2 + 2 c_f \veps^{-2} \|e_n^k\|_{L^3(\T^d)}^3
+ \veps^{-2} |\cC_n^k|_{-s}^2.
\end{split}\]
In an inductive argument we may assume that 
$\veps^{-2} \|e_n^k\|_{L^3(\T^d)}^3 \le c \|e_n^k\|^2$
and use the discrete Gronwall lemma to obtain an error estimate that depends exponentially
on the principal eigenvalue $\l^k$. Hence, if $\l^k$ is uniformly bounded from 
below the resulting error estimate depends only algebraically on $\veps^{-1}$. 
More generally, it suffices to assume that a discrete time integral of 
$\l^k$ is uniformly bounded from below. This allows us to cover large classes
of evolutions including topological changes. 

\section{Numerical Examples}\label{s:numerics}
In this Section, we present several numerical examples.  
In Section~\ref{s:pp} we discuss 
the approximation of the fractional Poisson problem. Section~\ref{s:imaging}
is devoted to image denoising problem. In Section~\ref{s:phase} we study features
of the fractional Allen--Cahn equation. We conclude with experiments
for the fractional Cahn--Hilliard equation in Section~\ref{s:phase_CH}.

\subsection{Approximation of the Poisson problem}\label{s:pp}
To construct a nonsmooth solution for the fractional Poisson problem 
we first let $w\in C(\T)$ be defined via 
\[
w(x) = \begin{cases} 
x, & x\le \pi, \\
2\pi -x, & x\ge \pi.
\end{cases}
\]
Since $w(0)=w(2\pi)$ we find for $k\neq 0$ with an 
integration-by-parts that 
\[
\hw_k = \int_0^{2\pi} w(x) e^{-\ikx} \dv{x} 
= \frac{1}{\i k} (1-e^{-\i kx}) \int_0^\pi e^{-\ikx}\dv{x} 
= \frac{1}{(\i k)^2} \big(1-(-1)^k\big)^2,
\]
i.e., $\hw_k= -4/k^2$ if $k$ is odd and $\hw_k=0$ if $k$ is even.
We have $\hw_0 = \pi^2$.  We then let $u\in C(\T^d)$ be
for $x\in \T^d$ defined via 
\[
u(x) = \prod_{i=1}^d w(x_i) - \frac{\pi^{2d}}{(2\pi)^d}.
\]
We have $\hu_k = \hw_{k_1} \cdots \hw_{k_d}$ if $k \neq 0$ 
and $\hu_0= 0$. We set $f = (-\Delta)^s u$, i.e., for $\kinZnd$
let $\hf_k = |k|^{2s} \hu_k$. Note that $f\in L^2(\T^d)$ if and
only if $s<1/2$. 
We choose the approximation $f_n = P_nf$ which is explicitly 
available here. The output 
for $s=1/2$ and $n=16$ is displayed in Figure~\ref{fig:frac_poisson_a}. 
In contrast to solutions of the classical Poisson problem we observe
here the occurrence of kinks in the solution. 

\begin{figure}[h!]
\includegraphics[width=.99\linewidth,height=3.0625cm]{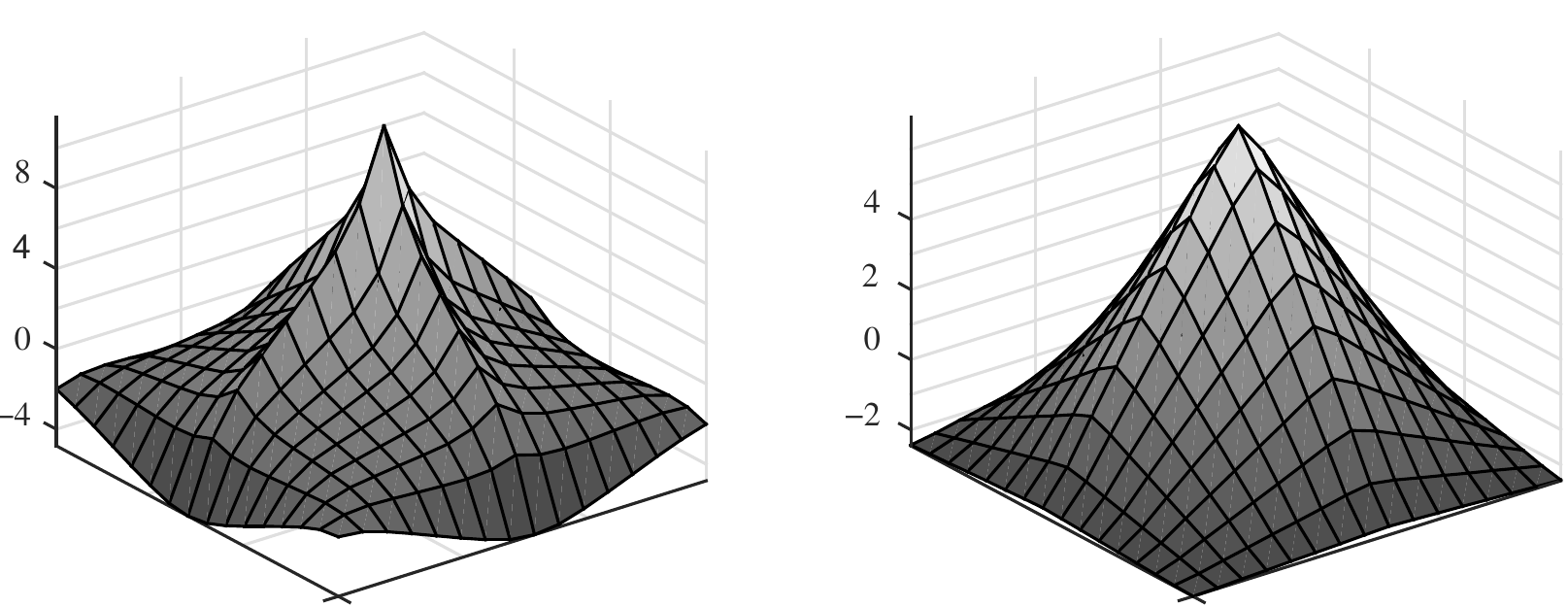}
\caption{\label{fig:frac_poisson_a} Functions $P_nf$ and $u_n$ 
for the fractional Poisson problem with $s=1/2$ and $n=16$.}
\end{figure}

\subsection{Fractional image denoising}\label{s:imaging}
The fractional Laplacian with $s<1/2$ is closely related to the
total variation norm (see Section~\ref{s:intro}) which motivates 
its application in 
image processing. Given a noisy image $g\in L^2(\T^d)$ we define
a regularized image $u\in H^{2s}(\T^d)$ via
\[
(-\Delta)^s u + \a (u-g) = 0.
\]
The fidelity parameter $\a$ penalizes the deviation of $u$ from
$g$ in the $L^2$ metric. Alternatively, this distance can be 
taken in the weaker metric of $H^{-1}(\T^d)$ which leads to
the equation 
\[
(-\Delta)^s u + \a (-\Delta)^{-1} (u-g) = 0,
\]
where we assume that $g$ has vanishing mean and look for
$u$ with vanishing mean. The results of two experiments are
displayed in the rows of Figure~\ref{fig:frac_poisson_d}. In the first 
experiment we set $s=0.42$, $\a=10$, and $n=1566$. 
In the second experiment we used $s=0.35$,  
$\a=5\times10^3$ and $n=256$. The first and second columns
display the original and noisy images, respectively. The third
and fourth columns show the results of $L^2$ and $H^{-1}$ fidelity.
We note that in the first example, where the additive noise 
is normally distributed with mean zero and standard deviation $0.15$,
the $L^2$-fidelity almost perfectly recovers the original image
reflecting the fact that for Gaussian noise this is statistically
the optimal choice.
In the second example where the noise is given by the nodal 
interpolant of the sinusoidal function 
\[
\xi(x_1,x_2) = 5\sin(20\pi x_1)\sin(20\pi x_2)
\]
we obtain better 
recovery using the $H^{-1}$-fidelity. We remark that it took 
0.2 sec and 0.006 sec to solve the first and the second problem 
in Matlab on a MacBook Pro with an 2.8 GHz Intel Core i7 processor 
(16 GB 1600 MHz DDR3 RAM).
 
\begin{figure}[h!]
\hspace*{-1.5cm}
\includegraphics[width=17cm,height=3.0cm]{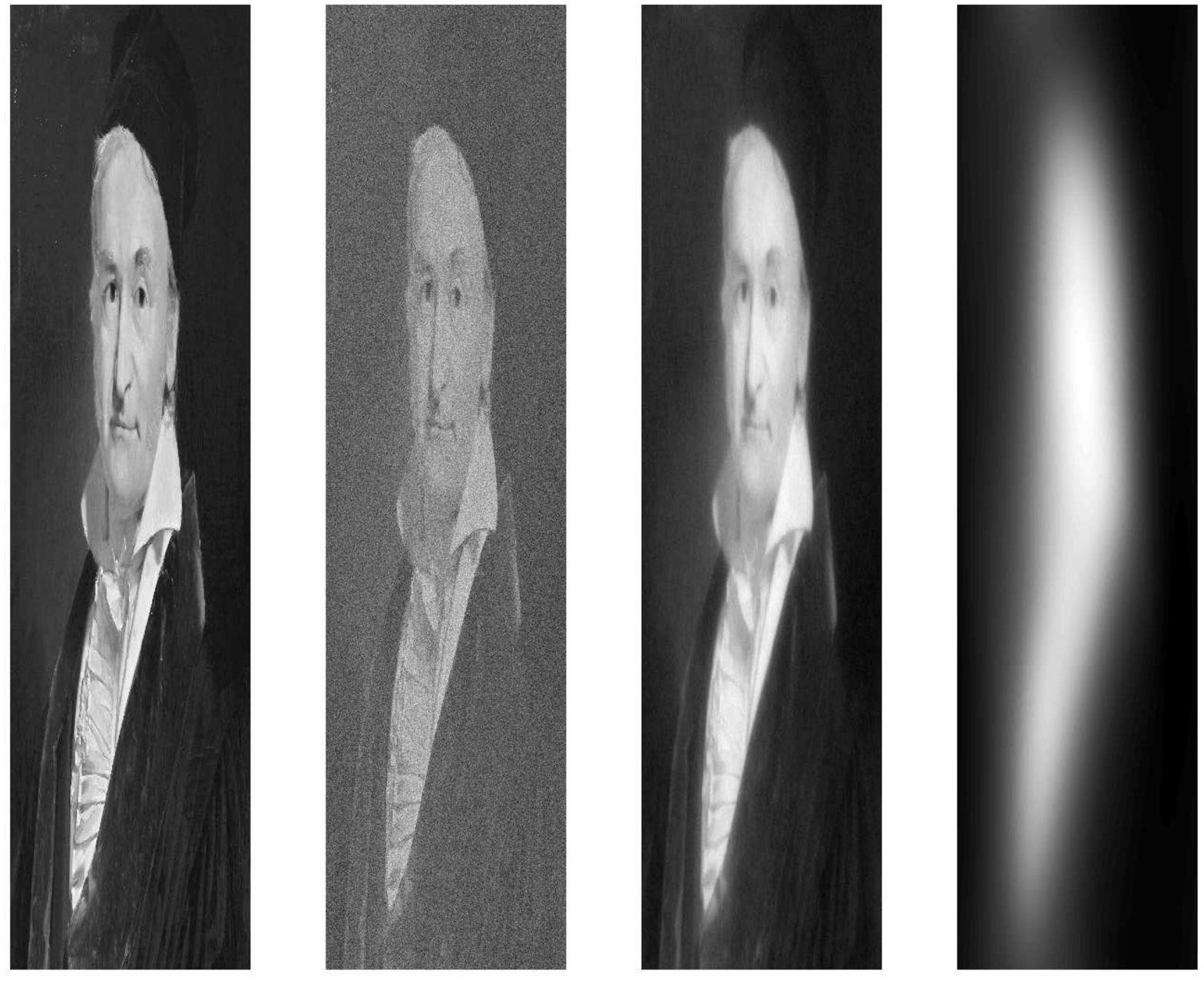}
\hspace*{-1.5cm}
\includegraphics[width=17cm,height=3.0cm]{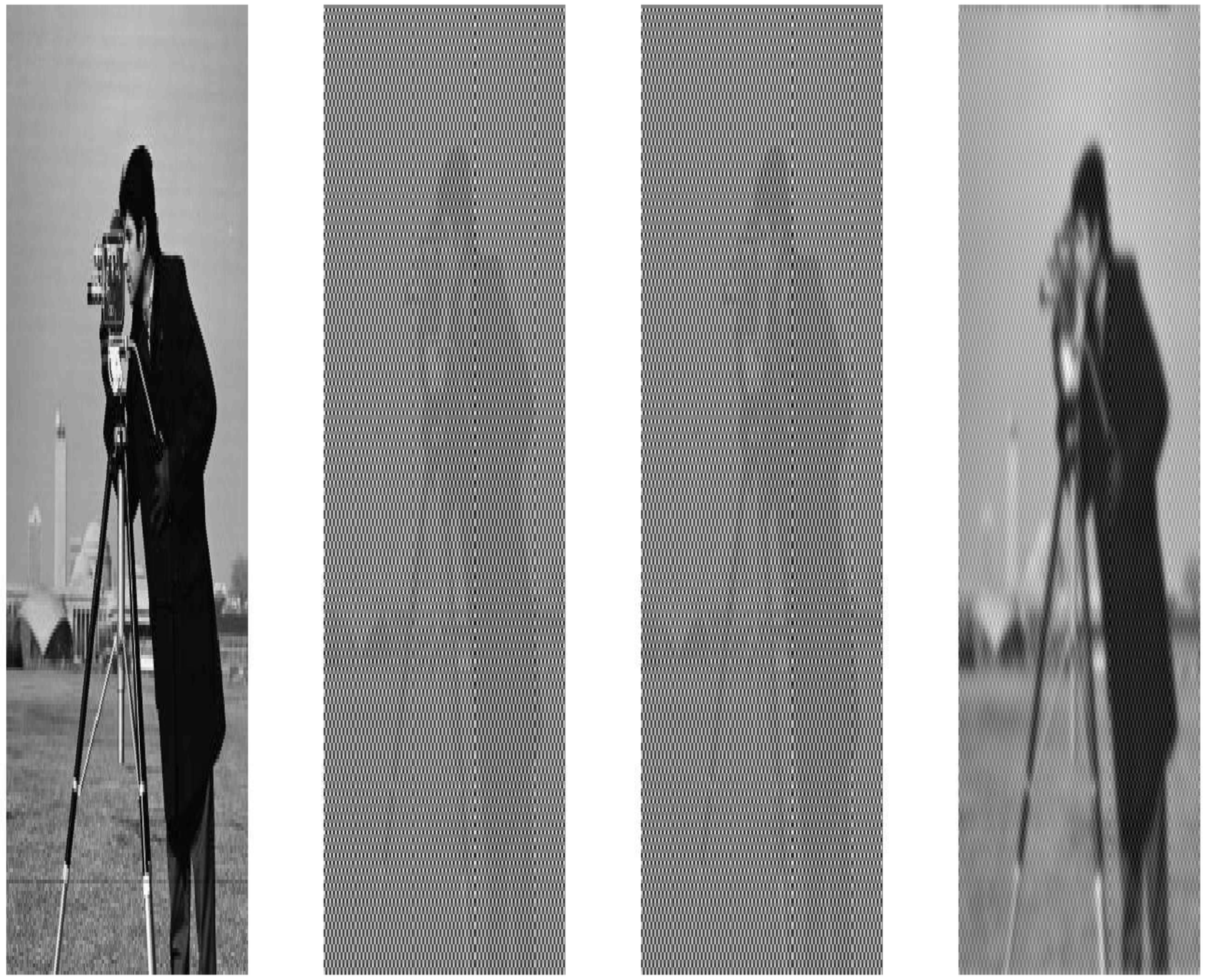}
\caption{\label{fig:frac_poisson_d} Original and noisy image,
regularized images for $L^2$ and $H^{-1}$ fidelity terms.
In the case of first example we have set $s=0.42$, 
$\a=10$, and $n=1566$. In the second 
example we have $s=0.35$, $\a=5\times10^3$ and $n=256$. 
The recovery is satisfactory using the $L^2$-fidelity in the first 
example. In the second example, recovery is better in case 
of the $H^{-1}$-fidelity term.}
\end{figure}

\subsubsection{Comparison between fractional and total variation}

We next compare the fractional and the total variation (TV) based models  
\eqref{eq:E} with the help of three examples. This comparison was carried out 
in Python which was specifically chosen due to the availability of SciKit-Image 
toolbox \cite{scikit-image}. 
In all the examples we assume that original image denoted by $o$ (without
noise) is known. For the fractional case we compute the optimal parameters 
$(s,\a)$ by solving a minimization problem: $\min_u$ $\|u-o\|^2$, subject to $u$ 
solving  
$
(-\Delta)^s u + \a (u-g) = 0 . 
$
We further assume that $(s,\a)$ lies in a closed convex set, i.e., 
$0.05 \le s \le 0.5$ and $1 \le \a \le 50$. We solve this optimization problem
using an in-built optimization algorithm in Python. The corresponding 
optimization problem for TV is solved for $\a$ using a genetic  algorithm. 

Our first example uses a picture of Gauss (cf.~Figure~\ref{f:circ}, top row). The
left image is the original image with $n = n_x = n_y = 1566$. Our second
example (cf.~Figure~\ref{f:circ}, middle row) uses a synthetic image 
with $n = n_x = n_y = 1500$. Our final 
example (cf.~Figure~\ref{f:circ}, bottom row)
is based on an in-built image from SciKit with different number 
of points in the $x$ and $y$ directions, i.e., $n_x = 400$, $n_y = 600$. We note that 
even though the approach discussed in Section~\ref{s:Image} assumes $n = n_x = n_y$ 
it is directly extended to this case where $n_x \neq n_y$. In the second
column (from the left), in all the examples, we have added a normally distributed 
noise with mean zero and standard deviation $0.15$, we denote the resulting noisy image
by $g$.

For the fractional case the optimal parameters are $(s,\a) = (0.49,50)$ for the first 
example and $(s,\a) = (0.5,50)$ for the second and third examples. On the other hand, for TV 
case the optimal parameters are $\a = 3.58$, $\a = 1.0$, and $\a = 8.74$, respectively. 
Using these  parameters we solve the corresponding image denoising problems using the 
fractional approach (third column from the left) and compare the results with the in-built TV 
algorithm from \cite{scikit-image}. We observe that the two approaches give comparable
results in all three examples. The fractional approach is significantly
cheaper as only two discrete Fourier transformations have to be computed. The CPU times 
for implementations of the fractional and TV approach in Matlab and Python, 
respectively, are provided in the caption of Figure~\ref{f:circ}. They show a 
reduction of the computing times by factors 10-100. We remark that certain aspects in the Chambolle--Pock algorithm
implemented in Scikit such as the specification of a suitable stopping
criterion and choice of step-sizes may lead to different results.  

\begin{figure}[h!]
\centering
\includegraphics[width=14cm,height=2.8cm]{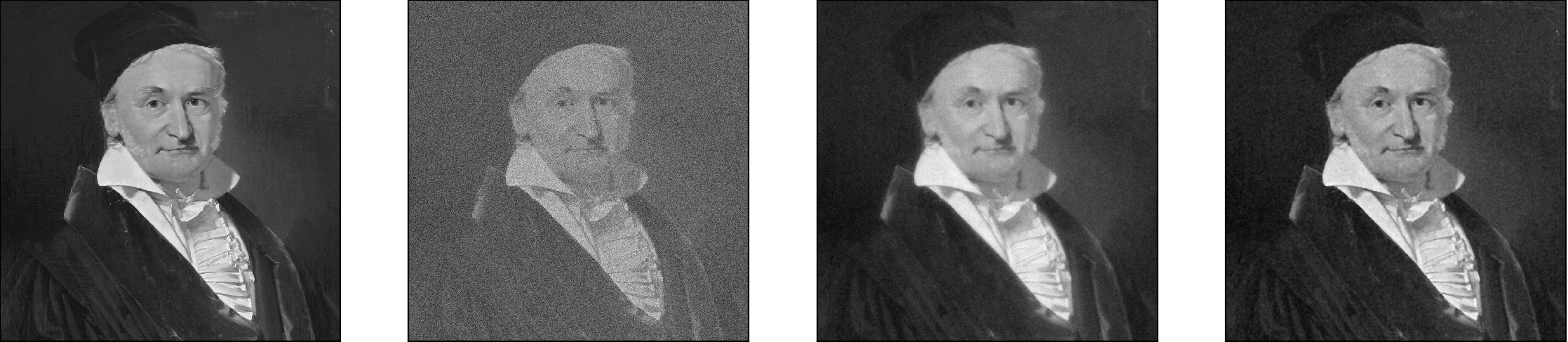}\\
\vspace{0.25cm}
\includegraphics[width=14cm,height=2.8cm]{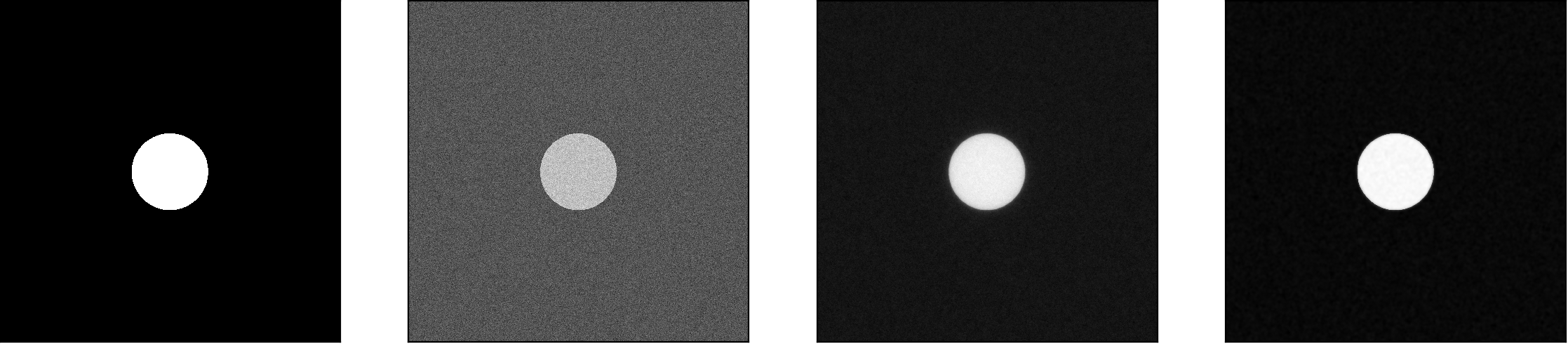}\\
\vspace{0.25cm}
\includegraphics[width=14cm,height=2.8cm]{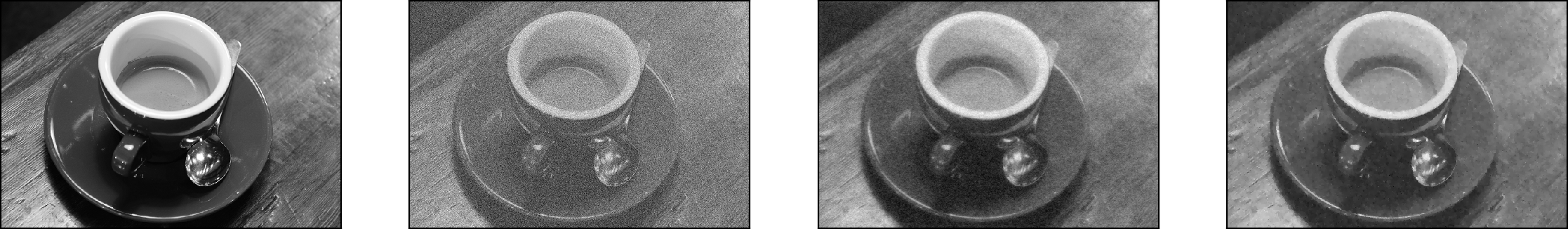}
\caption{\label{f:circ}Original and noisy image,
regularized images for $L^2$ fidelity with fractional and total variation (TV) approaches. 
The optimized parameters $s$ and $\a$ are given as follows - 
Example 1: $(s,\a) = (0.49,50)$ (fractional) and $\a = 3.58$ (TV); 
Example 2: $(s,\a) = (0.5,50)$ (fractional) and $\a = 1$ (TV); 
Example 3: $(s,\a) = (0.5,50)$ (fractional) and $\a = 8.74$ (TV).
The last two columns are corresponding reconstructions for 
the fractional and TV method. While the results are comparable the computing times
differ significantly: 
Example 1: 0.2 sec (fractional) and 2.1 sec (TV); Example 2: 
0.1 sec (fractional), 6.1 sec (TV); Example 3: 0.01 sec (fractional) and 0.2 (TV).}
\end{figure}

\subsection{Fractional Allen--Cahn equation}\label{s:phase}
We consider the fractional Allen-Cahn equation \eqref{eq:fracAC}
with a given initial function $u_0\in L^2(\T^d)$. The function~$f$ 
is the derivative of a double well potential~$F$
with quadratic growth, i.e., 
\[
F(u) = \frac12 \frac{(1-u^2)^2}{1+u^2}, \quad
f(u) = u - \frac{4 u}{(1+u^2)^2},
\]
which leads to linear systems of equations in our semi-implicit
time discretization. 
Snapshots of the evolutions with $\widetilde\veps = 1/8$, $n=512$, and
$\Delta t = 1/100$ at $t=1$, $t=4$, $t=12$, and $t=20$ (rowwise) are shown in 
Figure~\ref{fig:frac_phase}. The first column corresponds to $s=1$,
second to $s=0.45$, third to $s=0.30$, and finally fourth to $s=0.15$.
Clearly the interface in case of a fractional model is sharper, however
the dynamics are slower.

\begin{figure}[h!]
\includegraphics[width=\textwidth,height=0.2\textwidth]{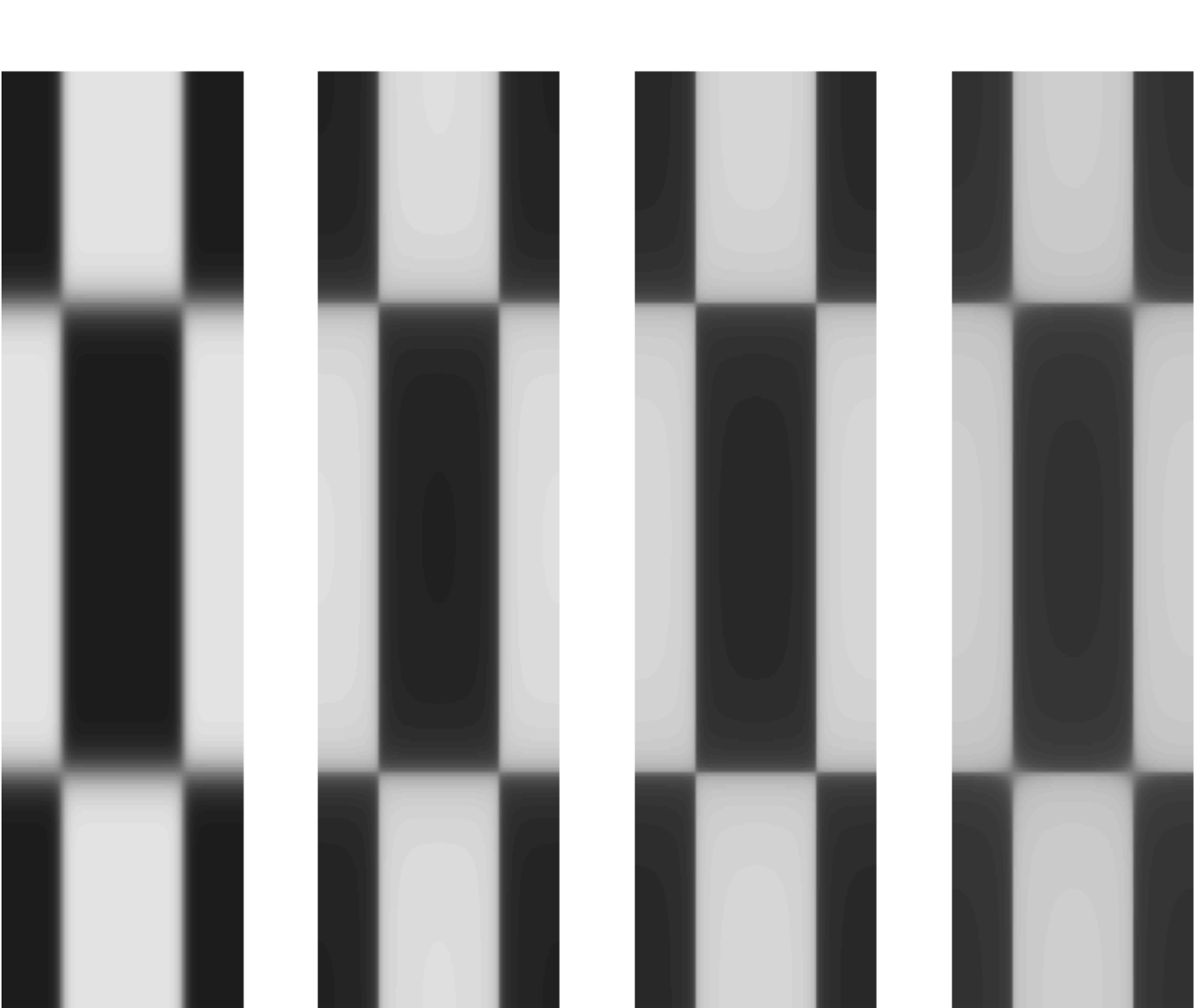}

\vspace*{0.5cm}
\includegraphics[width=\textwidth,height=0.2\textwidth]{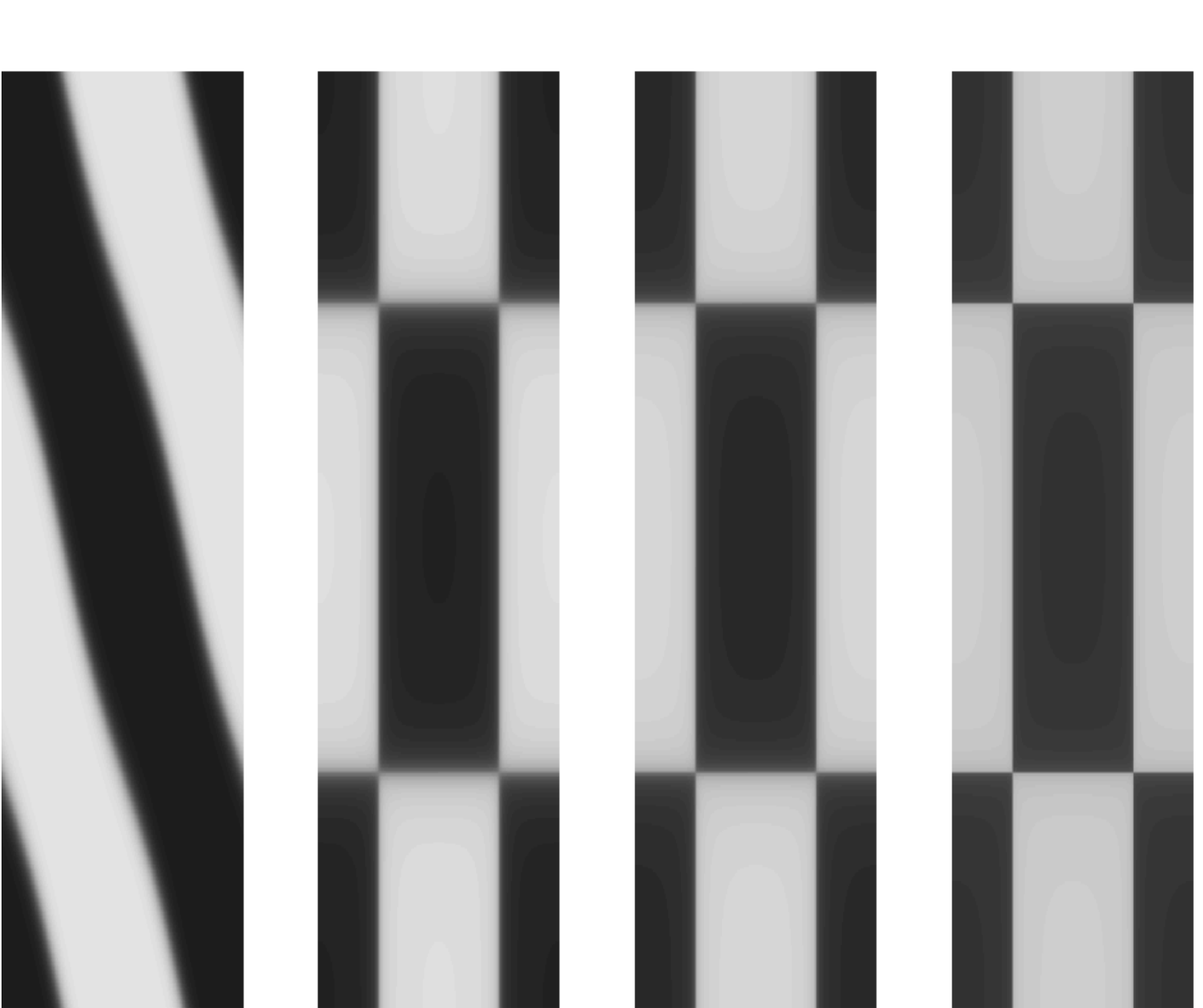}

\vspace*{0.5cm}
\includegraphics[width=\textwidth,height=0.2\textwidth]{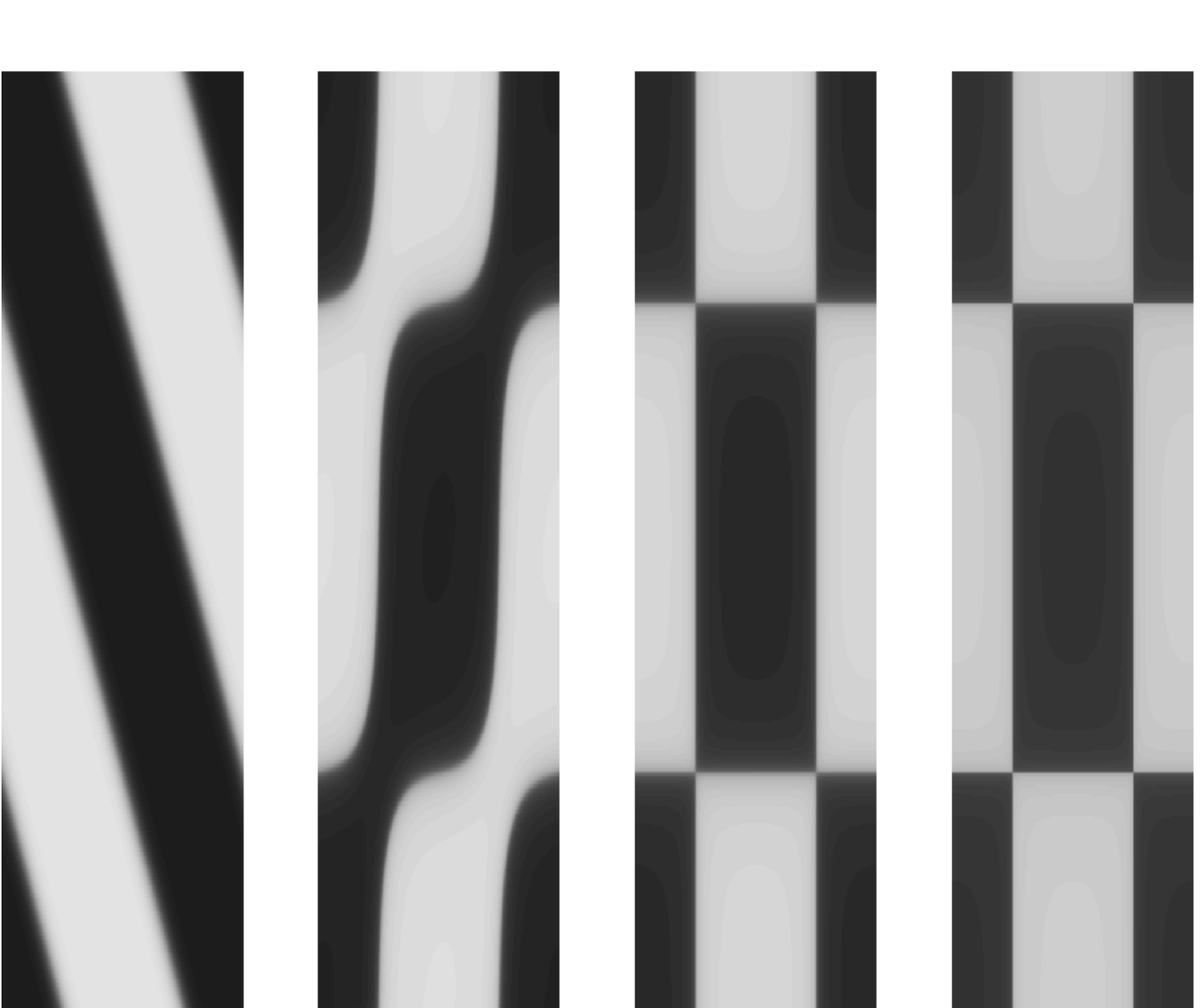}

\vspace*{0.5cm}
\includegraphics[width=\textwidth,height=0.2\textwidth]{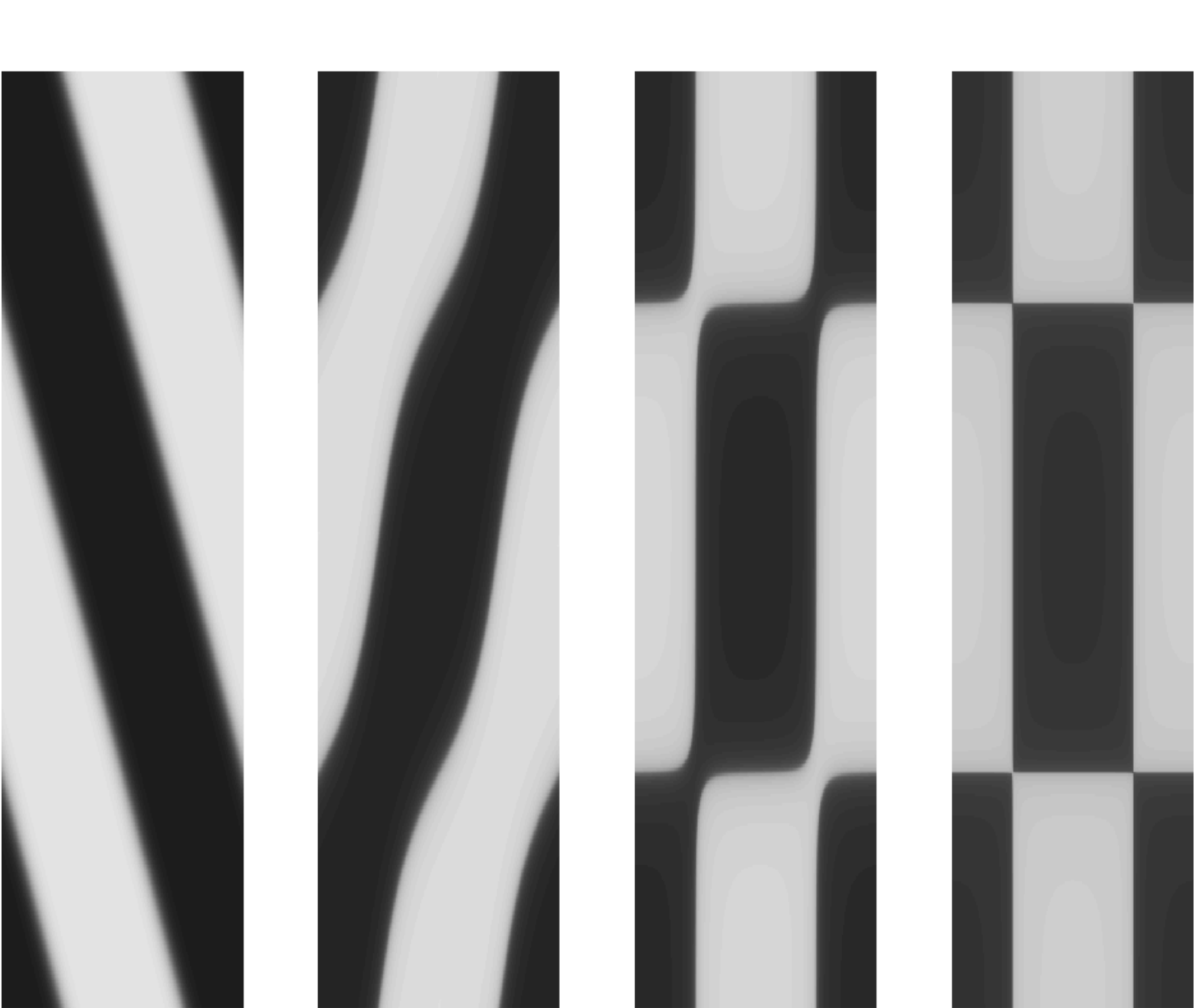}
\caption{\label{fig:frac_phase} Snapshots of Allen--Cahn evolutions with 
$n=512$ at times $t=1$, $t=4$, $t=12$, and $t=20$ (rowwise), respectively. 
Columns represent $s=1$, $s=0.45$, $s=0.30$, and $s=0.15$, respectively. 
In all cases we have set $\widetilde\veps= 1/8$.}
\end{figure}

\subsection{Fractional Cahn-Hilliard equation}\label{s:phase_CH}
We next study the fractional Cahn--Hilliard equation specified in 
\eqref{eq:fracCH} defining $F$ and $f$ as in Subsection~\ref{s:phase}.
In a first experiment we define the initial condition $u_0$ as 
\[
 u_0(x_1,x_2) = \begin{cases}
 1  &  \text{if } (x_1-\frac{2\pi}{3})^2+(x_2-\pi)^2<(\frac{\pi}{3})^2 \\
    &  \mbox{ or }  (x_1-\frac{4\pi}{3})^2+(x_2-\pi)^2<(\frac{\pi}{3})^2, \\
-1  & \mbox{otherwise}.
\end{cases}
\] 
Snapshots of the evolutions with $\widetilde\veps = 1/8$, $n=512$, and
$\Delta t = 1/100$ at $t=0.25$, $t=0.50$, $t=2$, and $t=3$ (rowwise) are 
shown in Figure~\ref{fig:frac_phase_CH} for $s=1$, $s=0.45$, $s=0.30$, and $s=0.15$ (columnwise) with 
$\alpha =1$ in all cases. 

\begin{figure}[h!]
\includegraphics[width=\textwidth,height=0.2\textwidth]{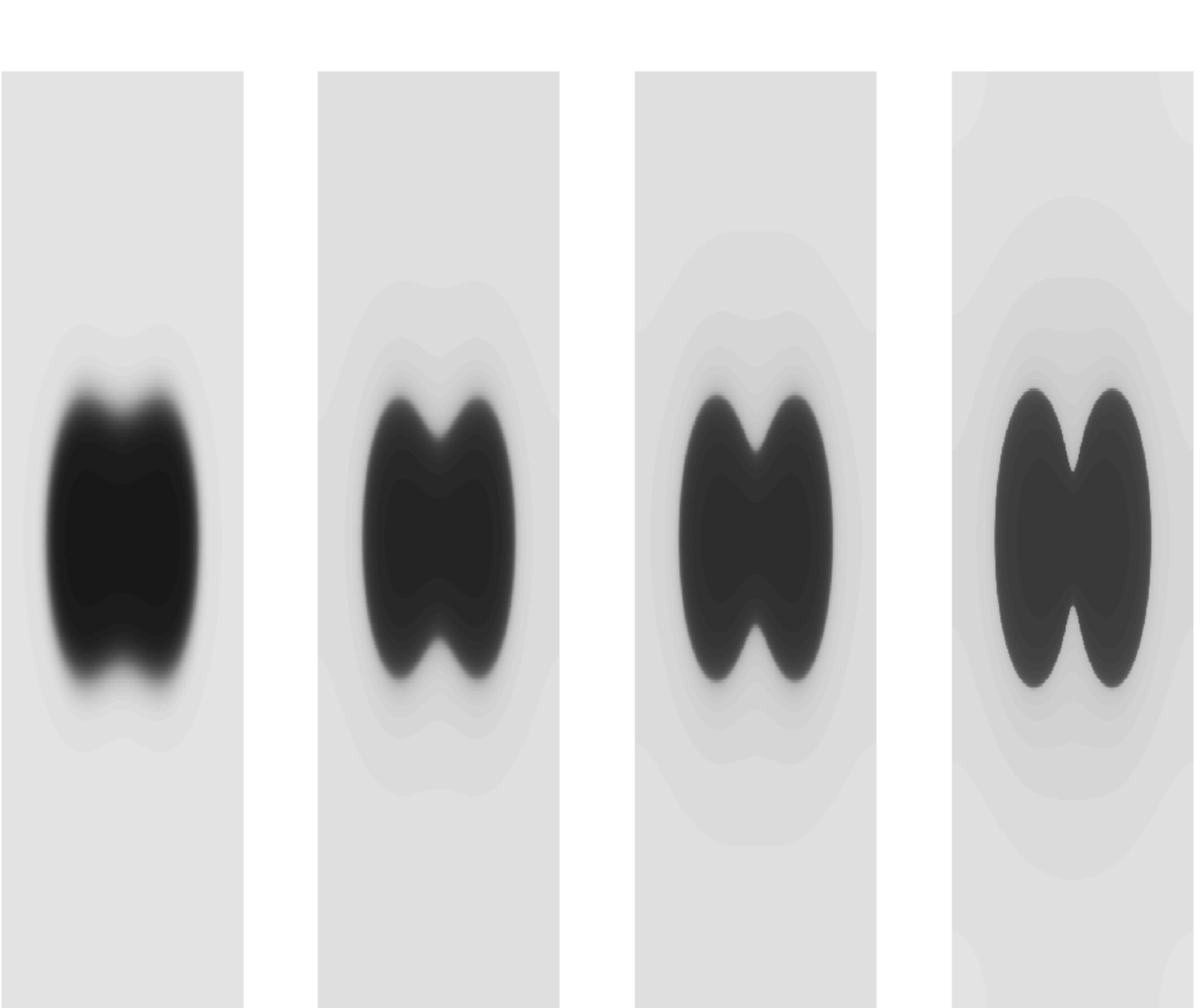}

\vspace*{0.5cm}
\includegraphics[width=\textwidth,height=0.2\textwidth]{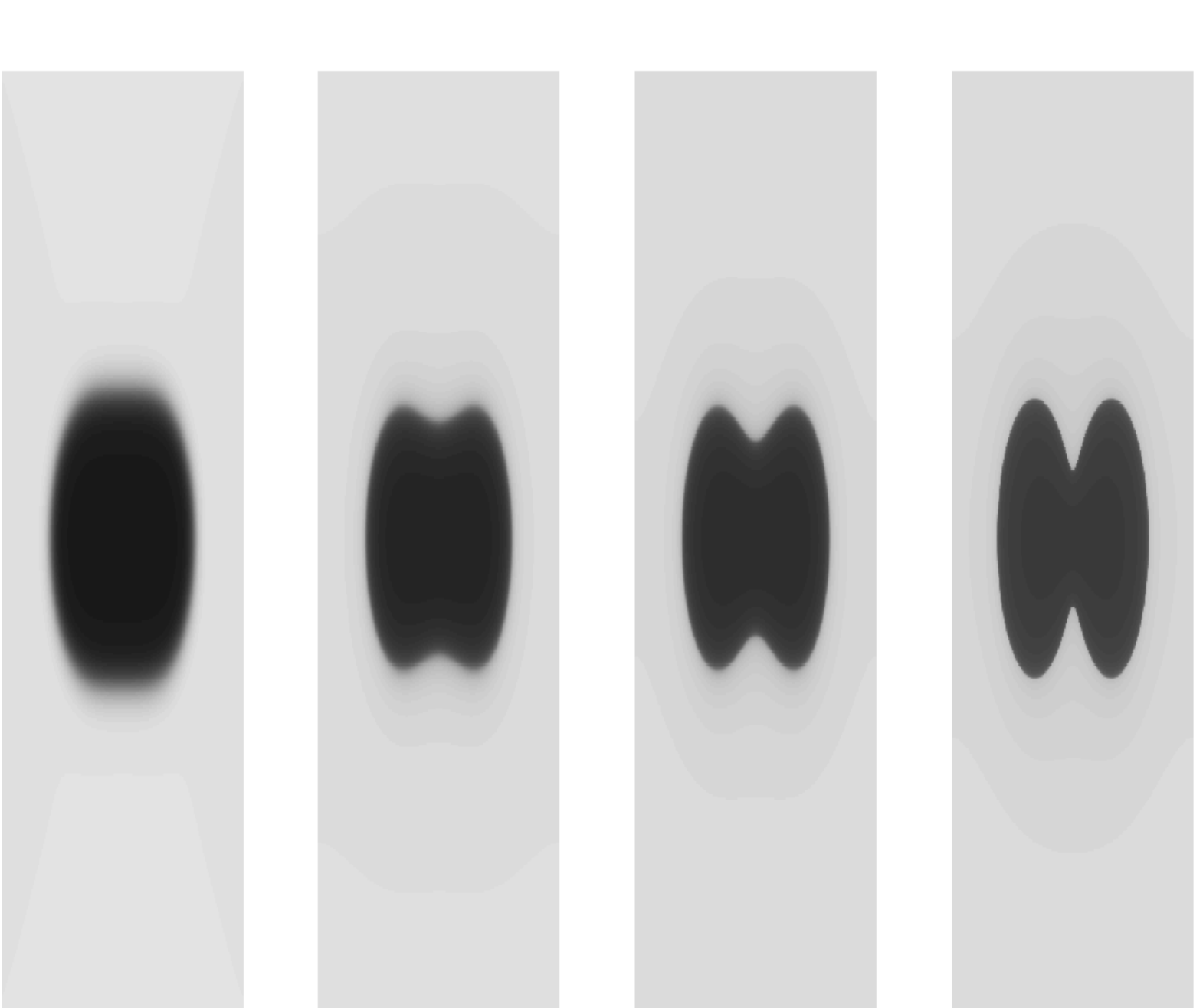}

\vspace*{0.5cm}
\includegraphics[width=\textwidth,height=0.2\textwidth]{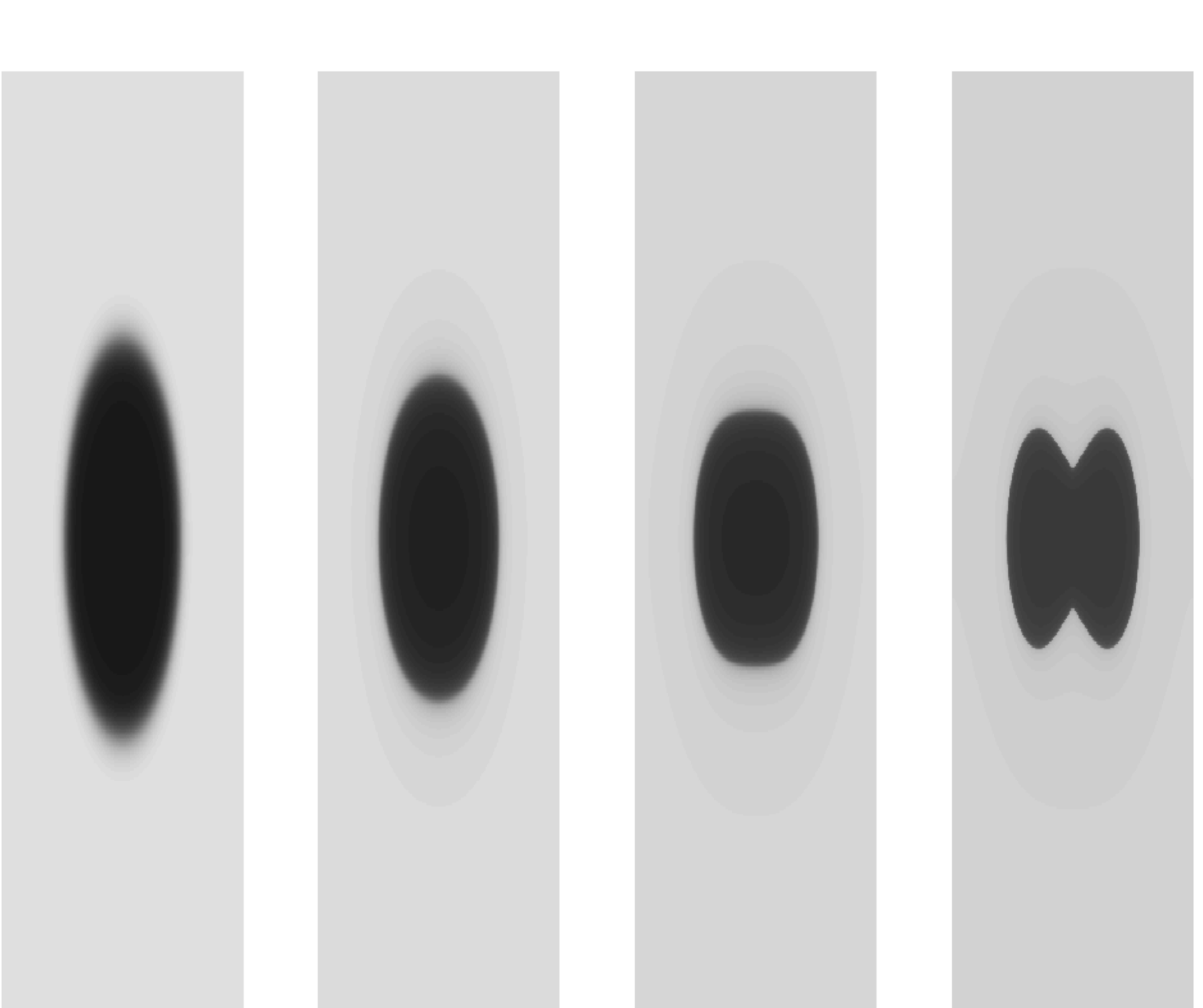}

\vspace*{0.5cm}
\includegraphics[width=\textwidth,height=0.2\textwidth]{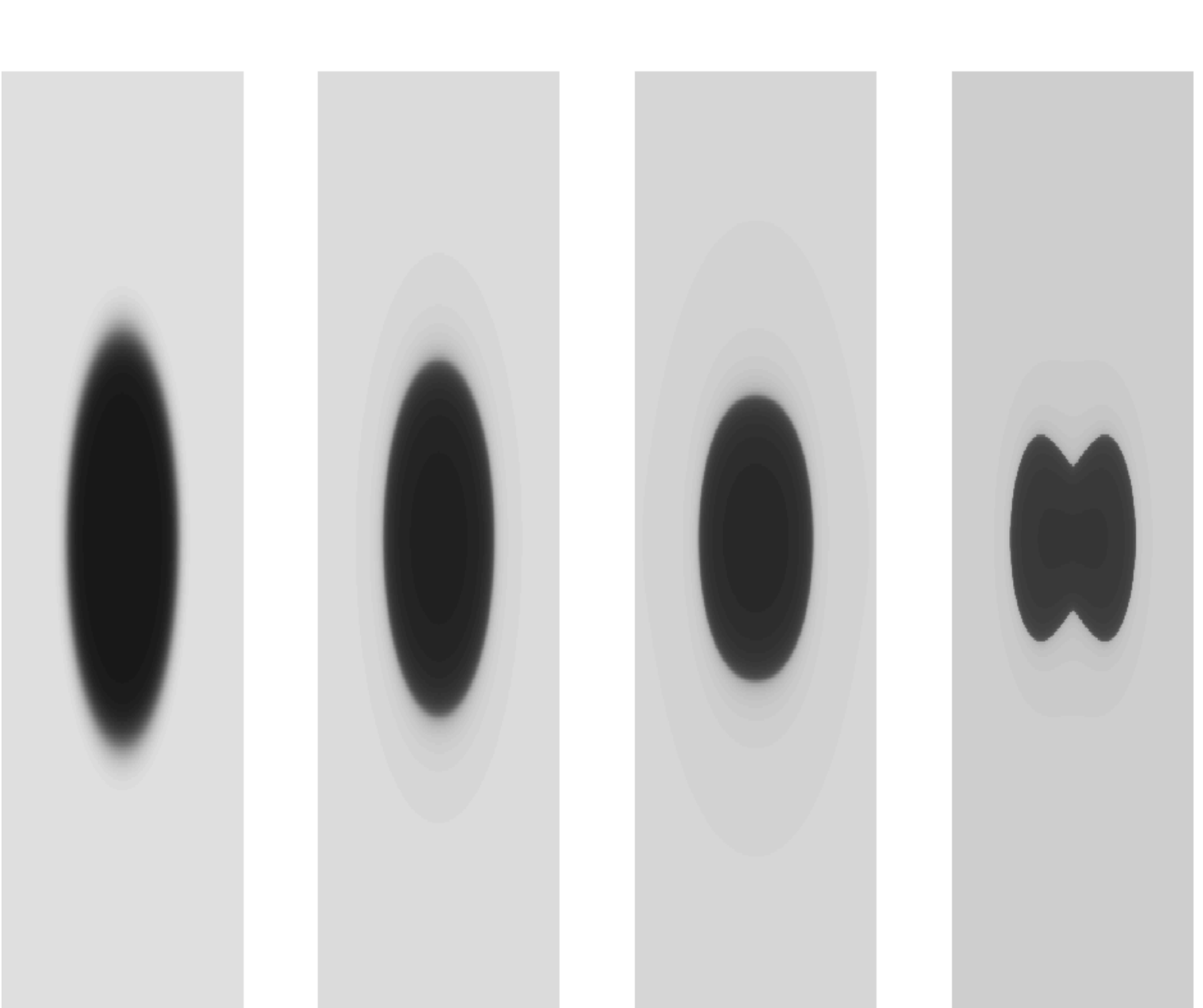}
\caption{\label{fig:frac_phase_CH}Snapshots of Cahn--Hilliard evolutions with $n=512$ 
where the rows correspond to $t=0.25$, $t=0.50$, $t=2$, and $t=3$, respectively. 
Columns represent $s=1$, $s=0.45$, $s=0.30$, and $s=0.15$, respectively. In all cases 
we have set $\widetilde\veps= 1/8$ and $\alpha=1$.}
\end{figure}

In a second experiment we focus on the coarsening dynamics
of the fractional phase field equation. As in \cite{ainsworthanalysis} the initial condition
is given by $u_0 = 2\phi-1 + \delta$ where $\delta$ is a random perturbation uniformly 
distributed in [-0.2,0.2]. 
Snapshots of the evolutions with $\phi=0.5$, 
$\widetilde\veps = 1/8$, $n=512$, and $\Delta t = 1/100$ at $t=0.25$, $t=0.5$, $t=1$, and $t=1.5$ 
(rowwise) are shown in Figure~\ref{fig:frac_phase_CH_2}. The first two columns correspond
to $\a=1$ with $s=1$ and $s=0.20$, respectively. The last two columns are obtained with
$\alpha=1/2$ with $s=1$ and $s=0.20$, respectively.

\begin{figure}[h!]
\includegraphics[width=\textwidth,height=0.2\textwidth]{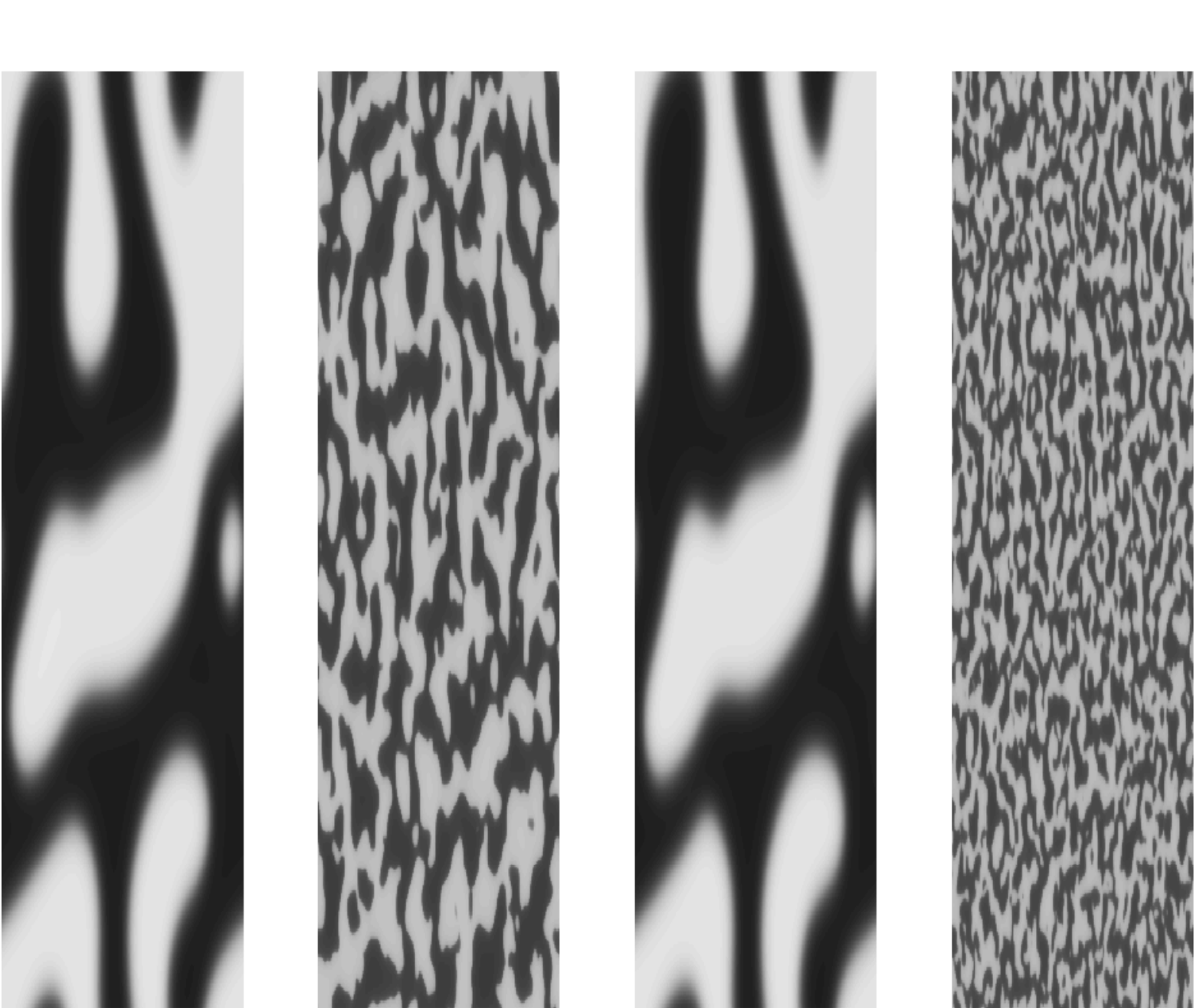}

\vspace*{0.5cm}
\includegraphics[width=\textwidth,height=0.2\textwidth]{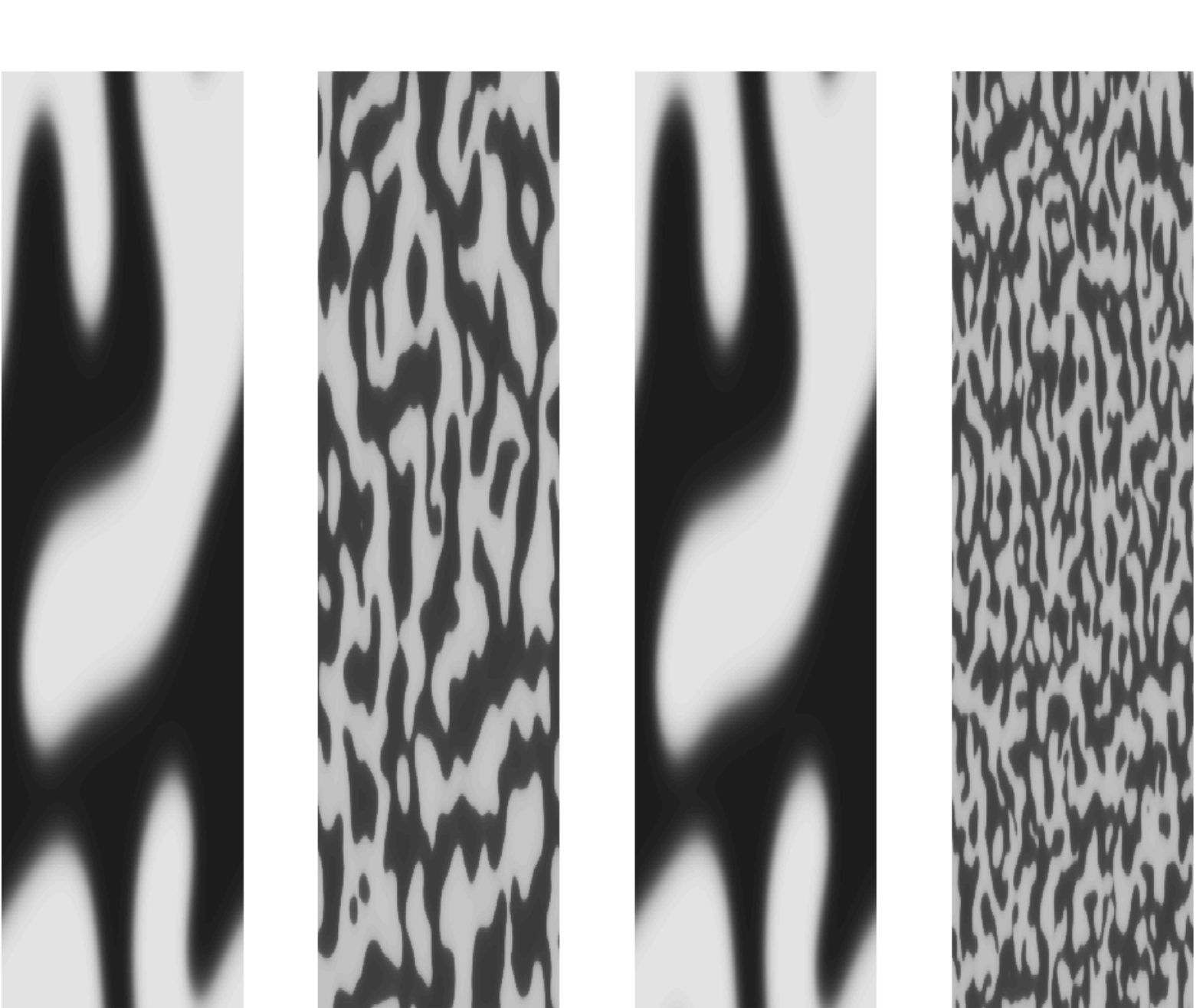}

\vspace*{0.5cm}
\includegraphics[width=\textwidth,height=0.2\textwidth]{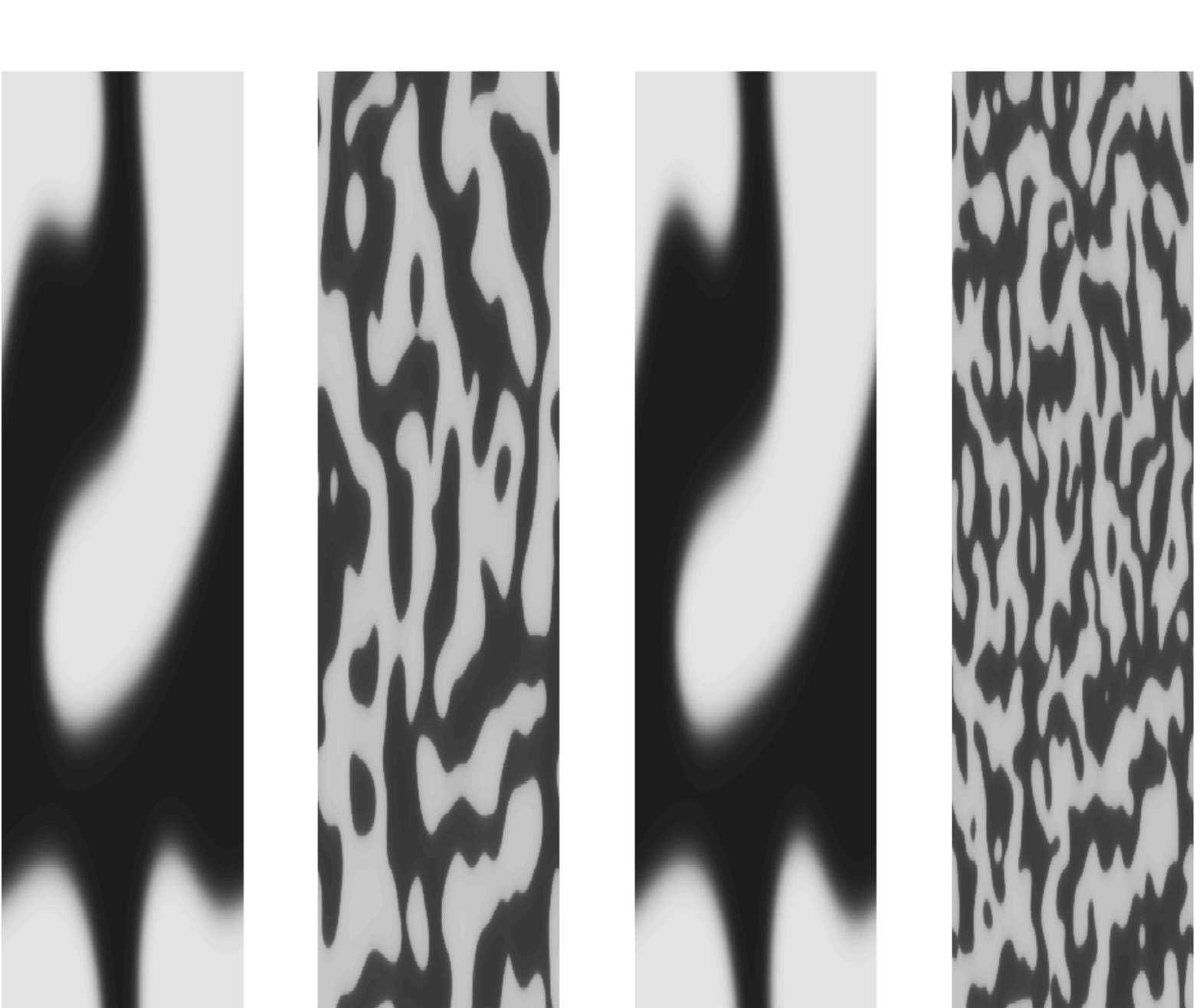}

\vspace*{0.5cm}
\includegraphics[width=\textwidth,height=0.2\textwidth]{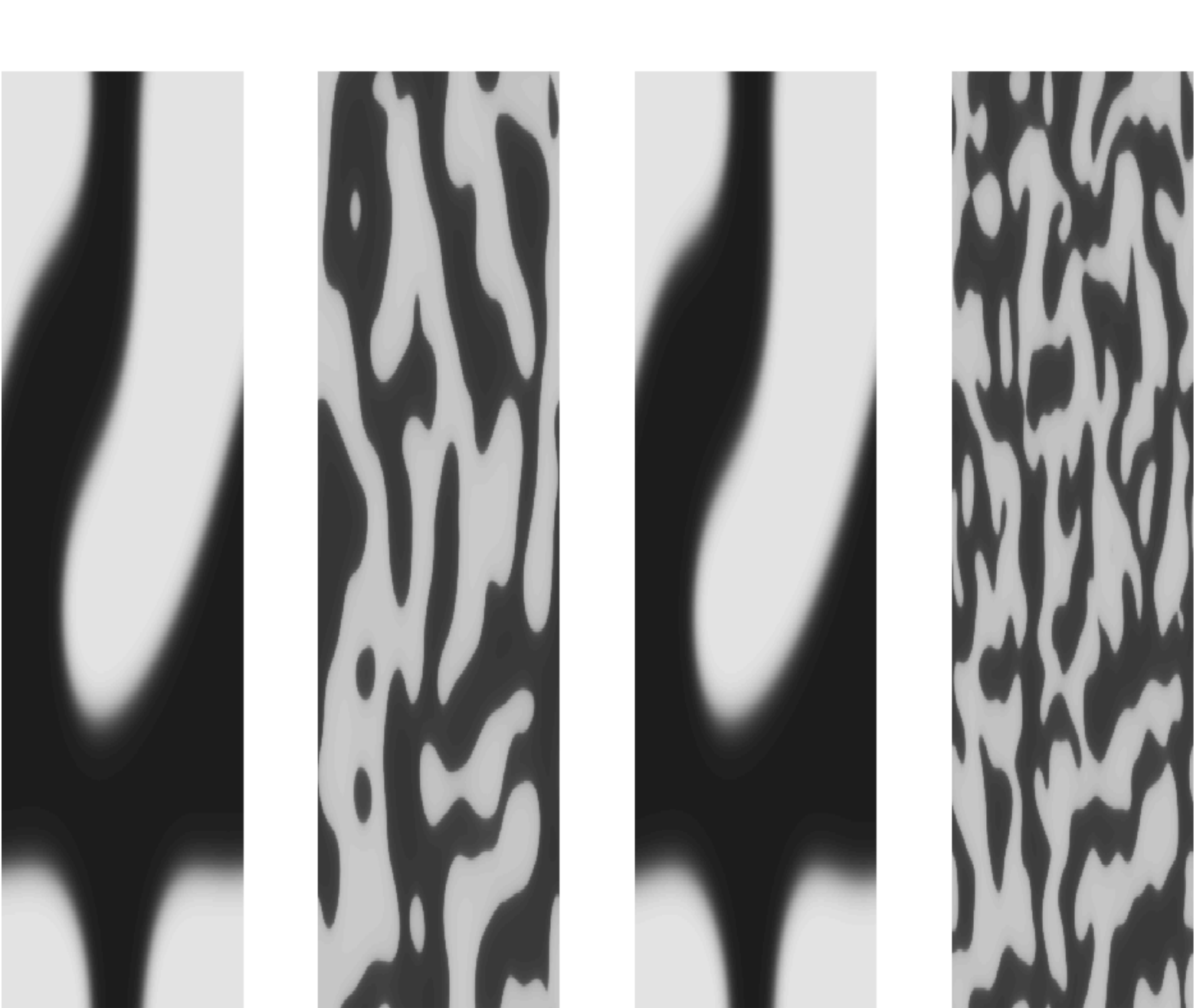}
\caption{\label{fig:frac_phase_CH_2}Snapshots of Cahn--Hilliard evolutions
at $t=0.25$, $t=0.5$, $t=1$, and $t=1.5$ for $\widetilde\veps= 1/8$ (rowwise). The first two 
columns represent $s=1$ and $s=0.20$, respectively with $\alpha=1$. The last two 
columns correspond to $s=1$ and $s=0.20$, respectively but here $\alpha=1/2$.}
\end{figure}

\section*{Acknowledgement}
We thank Pablo Stinga for stimulating discussions. We also thank Diego Torrejon 
for help with Python.

\section*{Funding}
The work of the first author is partially supported by NSF grant DMS-1521590.

\bibliographystyle{plain}
\bibliography{biblio}

\end{document}